\documentclass[reqno,11pt]{amsart}
\usepackage{amsmath,amssymb,amsthm,graphicx,a4wide,enumerate,url}
\usepackage[small,bf]{caption}
\usepackage{color}
\usepackage{cleveref}
\usepackage{wrapfig}
\usepackage{tikz}
\usetikzlibrary{calc,patterns,angles,quotes}
\usepackage{todonotes}
\theoremstyle{plain}
\newtheorem{theorem}{Theorem}
\newtheorem{lemma}[theorem]{Lemma}
\newtheorem{method}{Method}
\newtheorem{assumption}{Assumption}
\theoremstyle{definition}
\newtheorem{definition}{Definition}
\theoremstyle{remark}
\newtheorem{remark}{Remark}

\renewcommand{\vec}[1]{\mathbf{#1}}

\begin{document}
\parskip.9ex

\title[Infinite Time Solutions of Numerical Schemes for Advection Problems]
{Infinite Time Solutions of Numerical Schemes for Advection Problems}

\author[A. Biswas]{Abhijit Biswas}
\address[Abhijit Biswas]
{Department of Mathematics \\ Temple University \\
1805 North Broad Street \\ Philadelphia, PA 19122}
\email{tug14809@temple.edu}
\urladdr{https://math.temple.edu/\~{}tug14809/}

\author[B. Seibold]{Benjamin Seibold}
\address[Benjamin Seibold]
{Department of Mathematics \\ Temple University \\
1805 North Broad Street \\ Philadelphia, PA 19122}
\email{seibold@temple.edu}
\urladdr{http://www.math.temple.edu/\~{}seibold}



\begin{abstract}
This paper addresses the question whether there are numerical schemes for constant-coefficient advection problems that can yield convergent solutions for an infinite time horizon. The motivation is that such methods may serve as building blocks for long-time accurate solutions in more complex advection-dominated problems. After establishing a new notion of convergence in an infinite time limit of numerical methods, we first show that linear methods cannot meet this convergence criterion. Then we present a new numerical methodology, based on a nonlinear jet scheme framework. We show that these methods do satisfy the new convergence criterion, thus establishing that numerical methods exist that converge on an infinite time horizon, and demonstrate the long-time accuracy gains incurred by this property.
\end{abstract}

\maketitle

\section{Introduction}
\label{sec:introduction}
The accurate numerical resolution of transport is a crucial ingredient in many computational mathematics/physics/engineering problems. This work focuses on the linear transport under a velocity field that is either given, or arises from other components in a multiphysics computation, such as: (i)~the evolution of structures and/or interfaces defined by a level set function $\phi(\vec{x},t)$ \cite{OsherFedkiw2002, NaveRosalesSeibold2010}; (ii)~the passive convection of density fields or marker distributions under a flow field \cite{VanLeer1977_2}; or (iii)~the free-streaming step in particle transport computations \cite{Brunner2002, Chandrasekhar1960} solved via fractional steps. Moreover, while some applications admit Lagrangian approximations such as particle methods \cite{Monaghan2005, FarjounSeibold2009}, this work focuses on the accurate capture of advection via Eulerian methods, i.e., fixed-grid approaches, possibly with adaptive mesh refinement (AMR) \cite{BergerOliger1984}.

Modern (and future) compute infrastructure, combined with adaptive methodologies (such as AMR), enable the numerical resolution and evolution of small structures (such as droplets in a multi-phase fluid flow simulation, represented by a level set function) over millions of time steps. And with most existing Eulerian methods, even if the structure (level set function) itself remains unchanged (the droplet is merely moved by a locally constant velocity field, i.e., simple transport), the numerical scheme produces a spurious slow change of shape of the solution over time. While in some applications (but not all), overlapping grids \cite{Henshaw1994} and related techniques can mitigate the problem, they do not fully remove the problem of slow deformation over time. Hence, a question of fundamental interest is: \emph{are there numerical methods that do not suffer from a slow deterioration of the numerical solution in time, but instead possess convergent infinite time limits?} --- with the rationale that such methods will be very accurate for long-time computations of transport.

To investigate this fundamental question, we start out at a conceptual level, by considering the constant-coefficient linear advection equation in one space dimension
\begin{equation}
\label{eq:advection_1d}
u_t+a u_x = 0
\end{equation}
on the periodic domain $x\in [0,1]$, with initial conditions $u(x,0) = u_0(x)$. This problem is to be understood as a model problem for the advection of a field quantity (or level set function), $u_t+\vec{a}\cdot\nabla u = 0$, under a locally constant ($\vec{a} = \text{const.}$) or slowly varying ($\vec{a}(\vec{x},t)$) velocity field; and the periodic boundary conditions stand as a proxy for a long evolution on a large or unbounded domain.

Despite the simplicity of \cref{eq:advection_1d}, developing numerical methods that accurately evolve the solution on a fixed grid is a non-trivial task. Of course, for \cref{eq:advection_1d} itself, there are methods that are truly exact, namely most standard linear schemes when the CFL number is chosen exactly 1, so that the solution is moved exactly from one grid point to the next in one time step. However, because such a choice is not an option is most practical problems (e.g., due to varying or unknown velocity fields), we explicitly exclude this trivial case and restrict to CFL numbers truly less than 1.

The weaker objective to design numerical schemes that exhibit \emph{little} deformation of the numerical solution over long times has been pursued, for instance via Fourier continuation methods \cite{albin2012fourier, bruno2010high}. Due to low dispersive errors, these methods achieve high accuracy for advection-dominated problems compared to traditional finite difference/volume methods. Yet, the methods are unable to generate convergent infinite time limits, for the simple reason that they are linear (see \cref{sec:linear_methods}).

To construct numerical methods that \emph{do} possess convergent infinite time limits, we employ the jet scheme methodology \cite{SeiboldRosalesNave2012}, equipped with a special nonlinear interpolant. Jet schemes are semi-Lagrangian fixed-grid approaches that achieve high order by tracking function values and derivatives along characteristic curves, defined via optimally local Hermite interpolation at the foot points of the characteristic curves.

This paper is organized as follows. In \cref{sec:convergence_and_the_problem}, we define the new notion of convergence and formulate the mathematical problem. In \cref{sec:linear_methods}, we show that linear methods cannot achieve convergent infinite time limits. Then in \cref{sec:jet scheme} we present the new nonlinear jet scheme and establish the numerical analysis used to prove that the method yields convergent infinite time limits. Moreover, it is demonstrated how the new method produces superior long-time accuracy relative to classical nonlinear schemes. We close with a broader discussion and outlook in \cref{sec:conclusions}.

\vspace{1.5em}
\section{Notion of convergence and formulation of the problem}
\label{sec:convergence_and_the_problem}
To study numerical schemes for \cref{eq:advection_1d} in terms of their long-time, or even infinite time, accuracy, a departure from established numerical analysis notions of convergence of numerical approximations is needed, as follows. Traditionally, one chooses a fixed finite final time $t_\mathrm{f}$, and considers a sequence of meshes, parameterized by a mesh-size $h$. On each mesh, with time step $\Delta t = \frac{\mu h}{a}$ (where $0<\mu<1$ is the fixed Courant-Friedrichs-Lewy (CFL) number \cite{CourantFriedrichsLewy1928}) one computes a numerical approximation $\boldsymbol{U}^{h,t_\mathrm{f}}$ at the fixed final time $t_\mathrm{f}$. The numerical scheme is then called \emph{convergent} if the error $\boldsymbol{\varepsilon}^{h,t_\mathrm{f}} = \|\mathcal{I}(\boldsymbol{U}^{h,t_\mathrm{f}})-\boldsymbol{u}^{t_\mathrm{f}}\|$ satisfies $\boldsymbol{\varepsilon}^{h,t_\mathrm{f}} \to 0$ as $h\to 0$, where $\boldsymbol{u}^{t_\mathrm{f}}$ is the true solution at time $t_\mathrm{f}$, the norm $\|\cdot\|$ is a suitable function space norm \cite{LeVeque2007}, and $\mathcal{I}(\boldsymbol{U}^{h,t_\mathrm{f}})$ extends $\boldsymbol{U}^{h,t_\mathrm{f}}$ to the full domain via a suitable interpolation (see \cref{subsec:qualitative_analysis} and \cref{subsec:two_jet_interpretations}).

The proposed new notion of convergence considers the error as a function of both $h$ and time $t_n$, where $t_n=n \Delta t$, $n \in \mathbb{N}$. For instance, given $h$ with $\frac{1}{h}=m \in \mathbb{N}$ and $t_n$, compute the numerical approximation (with $\Delta t = \frac{\mu h}{a}$) at time $t_n$, denoted $\boldsymbol{U}^{h,t_n}$. Then define the error $\boldsymbol{\varepsilon}^{h,t_n} = \|\mathcal{I}(\boldsymbol{U}^{h,t_n})-\boldsymbol{u}^{t_n}\|$, using the same norm as above. With this bivariate error, the limit $h\to 0$ with $nh$ fixed (i.e., point-wise convergence in time) recovers the traditional notion of convergence. However, one can also require more restrictive notions on the limit $h\to 0$. The ideal situation, which is the strongest notion of convergence that we establish here, is the commuting limits property
\begin{equation}
\label{commuting_limits}
\lim_{h\to 0}\limsup_{n\to\infty} \boldsymbol{\varepsilon}^{h,t_n}
= \limsup_{n\to\infty}\lim_{h\to 0} \boldsymbol{\varepsilon}^{h,t_n}\;.
\end{equation}
Here the limits $\limsup_{n\to\infty}$ are to be understood in a moving frame of reference $x-at$, in line with the true solution of \cref{eq:advection_1d} which is $u(x,t_n) = u_0(x-at_n)$, or $u(x,t_n) = u_0(\textrm{mod}(x-at_n,1))$ on the periodic domain $[0,1]$. Note that the $\limsup_{n\to\infty}$ is used because a common occurrence in numerical methods is that they produce a quasi-periodic temporal evolution of the error. So an error limit may not exist in the strict sense, and the $\limsup$ is the right measure of the maximum error over time in such a quasi-periodic case.

Our objective is to have a numerical scheme that produces accurate numerical solutions at any time step. In the case of \cref{eq:advection_1d}, we want a numerical scheme that mimics the translation property of the advection equation, i.e., the method correctly shifts a numerical approximation from some instance in time to the next time step. Given that on the periodic domain $x\in [0,1]$ the solution returns to its initial configuration every $\frac{1}{a}$ times, a numerical solution is said to be a \emph{fixed point} if it agrees at time $t+\frac{n}{a}$, for some integer $n$, with its state at time $t$. Moreover, given an interpolant $\mathcal{I}$ that defines a numerical solution on the whole domain (see above), a numerical approximation is said to be a \emph{single step fixed point} if the scheme precisely moves it from one time step to the next time step. Hence, having a single step fixed point for all steps is a special case of having a fixed point solution. Finally, we call a scheme an \emph{exact method} if any smooth solution can be approximated by a convergent sequence of fixed points (see \cref{subsec:numerical_analyis_jet_scheme} for establishing these properties).

We see that exact schemes are characterized by the commuting limits property in \cref{commuting_limits}. If a numerical scheme is convergent in the traditional sense, then for any fixed time with $nh$ fixed, $\lim_{h\to 0} \boldsymbol{\varepsilon}^{h,t_n}=0$ by definition. So $\limsup_{n\to\infty}\lim_{h\to 0} \boldsymbol{\varepsilon}^{h,t_n}=0$. In order to check whether a convergent method is \emph{exact}, it suffices to verify that
\begin{equation*}
\lim_{h\to 0}\limsup_{n\to\infty} \boldsymbol{\varepsilon}^{h,t_n}
= 0\;.
\end{equation*}

\vspace{1.5em}
\section{Linear methods}
\label{sec:linear_methods}
The first question is: can linear schemes have the commuting limits property \cref{commuting_limits}, or even more: do linear schemes exist that precisely shift the numerical solutions (to constant-coefficient linear advection problems) from one time step to next time step? Moreover, if such numerical solutions exist, then are those exact solutions rich enough (i.e., form a dense set) to allow the approximation of any arbitrary initial conditions from a class of functions?

Given a grid with mesh size $h$, let $\boldsymbol{U}^{h,t_n}$ denote the state vector of the numerical approximation at the $n$-th time step. There is a wide variety of numerical methodologies, for example: finite difference methods \cite{LeVeque2007} choose the components of $\boldsymbol{U}^{h,t_n}$ to be approximations to the solution at grid points; finite volume methods \cite{Godunov1959, LeVeque2002} use cell averages; discontinuous Galerkin methods \cite{CockburnKarniadakisShu2000, CockburnShu2001} use polynomial moments; jet schemes \cite{SeiboldRosalesNave2012, ChidyagwaiNaveRosalesSeibold2012} use point-wise function values and derivatives. However, what is common in \emph{linear} methods (i.e., schemes without limiters \cite{VanLeer1973} or other nonlinear components), is that the effects of many steps are fully captured by understanding a single step of the method. Here we consider two frameworks to understand the long time solutions of linear methods. Given $h$ and $\Delta t$ related via a CFL number, we either take a finite domain with periodic boundary conditions (to demonstrate numerical examples), or we consider an infinite domain (to carry out a Von Neumann stability analysis \cite{LeVeque2007} without boundary condition artifacts). We consider the recurrence relation
\begin{equation}
\label{recurrence_relation}
     U_{j}^{n+1}=g(\xi) U_j^{n}
\end{equation}
to obtain numerical solutions $\boldsymbol{U}^{h,t_n}=\left\{ U_{j}^{n} \right\}_{j=-\infty}^{\infty}$ on an infinite grid at time $t_n$, where the \emph{amplification factor} $g(\xi)$ is a function of the wave number $\xi$. Note that, for notational economy, we here omit the superscript $h$ in the components of $\boldsymbol{U}^{h,t_n}$ and write $n$ instead of $t_n$. For a finite grid with periodic boundary conditions, $X_m = \left\{x_{0}, x_{1}, \ldots, x_{m-1}\right\}$, where $x_{j} = j h$ for $j = 0,1,\dots, m-1$ and $h = \frac{1}{m}$, using the analogous notation, the update rule of the method can be written as
\begin{equation*}
\boldsymbol{U}^{h,t_{n+1}} = M_h \boldsymbol{U}^{h,t_n}\;.
\end{equation*}
Now $g(\xi)$ assumes only discrete values, and they are exactly the eigenvalues of the matrix $M_h$. We also assume that the scheme uses the same stencil for each grid point. Then the scheme is defined by a vector $\boldsymbol{c}=(c_0,c_1, \ldots,c_{m-1}) \in \mathbb{R}^m$, and $M_h$ becomes a circulant matrix. The scheme's behavior is then determined by the spectral properties of $M_h$. The matrix $M_h$ and its $\mathrm{j}\text{-th}$ eigenvector are given by
\begin{equation*}
     M_h =\begin{pmatrix}
    {c_{0}} & {c_{1}} & {c_{2}} & {\cdots} & {c_{m-1}} \\
    {c_{m-1}} & {c_{0}} & {c_{1}} & {c_{2}} & {\cdots} \\
    {c_{m-2}} & {c_{m-1}} & {c_{0}} & {\cdots} & {} \\
    {\ddots} & {\ddots} & {\ddots} & {\ddots} & {\ddots} \\
    {c_{1}} & {c_{2}} & {\cdots} & {c_{m-1}} & {c_{0}}
    \end{pmatrix}
    \ \ \text{and} \ \
    v_j=
    \begin{pmatrix}
    {\omega_{m}^{0 j }} \\
    {\omega_{m}^{1 j}} \\
    {\omega_{m}^{2 j}} \\
    {\vdots} \\
    {\omega_{m}^{(m-1) j}}
    \end{pmatrix}\;,
\end{equation*}
corresponding to the eigenvalue $\lambda_{j}=\sum_{k=0}^{m-1} c_{k} \omega_{m}^{k j}$ for $j = 0,1,\ldots, m-1$. Here $\omega_{m}=e^{\frac{2\pi i}{m}}$ and $\omega_{m}^{0},\omega_{m}^{1}, \dots, \omega_{m}^{m-1}$ are the $m$-th roots of unity. Circulant matrices are diagonalizable; let $M_{h}=F_{h}\Lambda_{h}F_{h}^{-1}$ with
$F_h=\left[v_{0} \mid v_{1} \mid \cdots \mid v_{m-1}\right]$ and
\sloppy $\Lambda_{h}=\operatorname{diag}\left(\lambda_0,\lambda_1, \ldots,\lambda_{m-1}\right)$. $F_h$ is a symmetric matrix with $F_{h}^{-1}=\frac{1}{m}F_{h}^{*}$. We obtain the numerical approximation after $n$ steps, $\boldsymbol{U}^{h,t_n} = (M_{h})^{n} \boldsymbol{U}^{h,t_0} = F_{h}(\Lambda_{h})^{n}F_{h}^{-1}\boldsymbol{U}^{h,t_0} = F_{h}(\Lambda_{h})^{n}\alpha^{h} = \sum_{j=0}^{m-1} v_j \alpha_j^{h} (\lambda_j)^n\;,$
where $\alpha^{h} = F_{h}^{-1}\boldsymbol{U}^{h,t_0} = \begin{bmatrix} \alpha_0^{h},\alpha_1^{h}, \ldots, \alpha_{m-1}^{h} \end{bmatrix}^{T}$.
All the components corresponding to $\lambda_j$ with $|\lambda_j|<1$ will decay to zero as $n\to\infty$, and all that remains are the components corresponding to $\lambda_j$ with $|\lambda_j| = 1$. Let $\mathcal{C} = \{j : |\lambda_j| = 1\}$, and for all $j\in \mathcal{C}$ write $\lambda_j = e^{\mathrm{i}\theta_j}, \ \theta_j \in [0,2\pi)$. Then the infinite time limit becomes
\begin{equation*}
\lim_{n\to\infty} \boldsymbol{U}^{h,t_n}
= \sum_{j\in \mathcal{C}} v_j \alpha_j^h e^{\mathrm{i}\theta_j n}\;.
\end{equation*}
Before we move on to providing a generic proof on an infinite grid, we demonstrate the long time solutions via three illustrative examples on a finite domain.

\subsection{Illustrative examples}
\Cref{fig:upwind} shows the (well-known) behavior of the basic upwind method, whose behavior is typical for many other schemes as well. All eigenvalues are truly inside the unit disc, except for one at $\lambda_1 = 1$ whose eigenvector corresponds to a constant function (a consequence of having a consistent scheme). Hence, as $n\to\infty$, all modes decay except for the constant mode, and the numerical solution always decays towards a constant function (the average of the initial condition), independent of $h$, thus $\lim_{h\to 0}\limsup_{n \to\infty}\boldsymbol{\varepsilon}^{h,t_n} = \textrm{const}$. Consequently, the scheme is not exact, and the limits $h\to 0$ and $\limsup_{n \to\infty}$ do not commute.

\begin{figure}[ht]
\begin{minipage}[b]{0.37\textwidth}
\includegraphics[width=\textwidth]{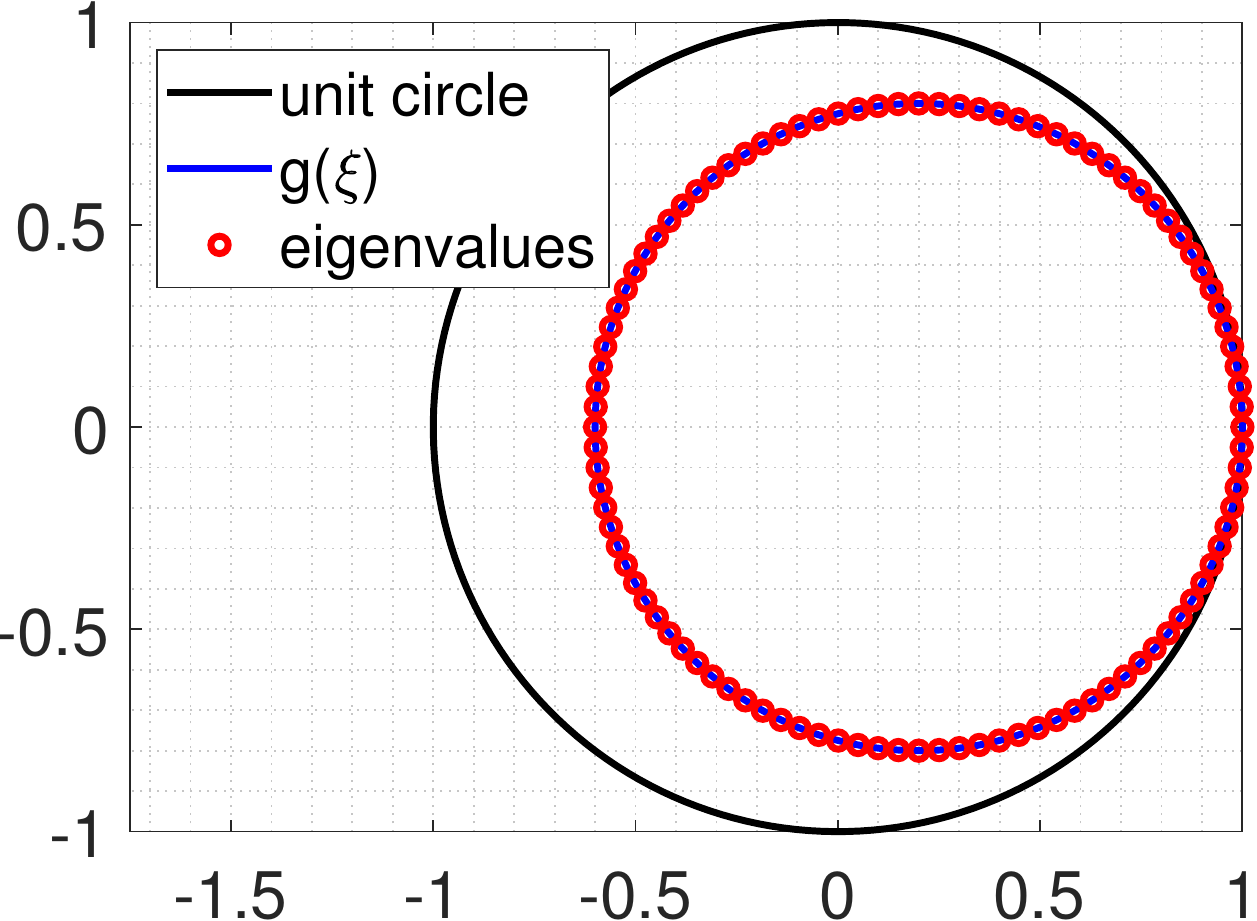}
\end{minipage}
\hfill
\begin{minipage}[b]{.57\textwidth}
\includegraphics[width=\textwidth]{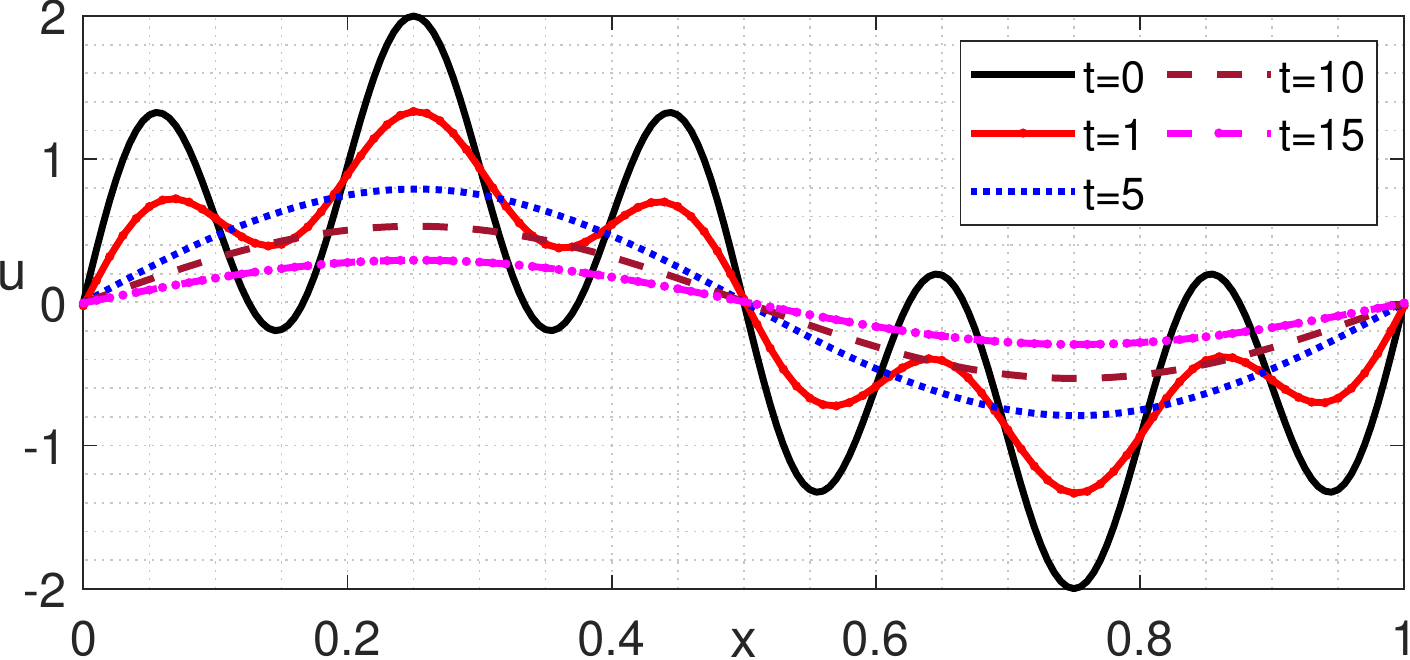}
\end{minipage}
\caption{Upwind method for $u_t+u_x=0$ with $\mu=0.8 \ (\mathrm{CFL})$ and $m=100$ grid points: Amplification factor $g(\xi)=(1-\mu)+\mu e^{-i \xi h}$, and eigenvalues (left) and numerical solution at different times (right). All the eigenvalues except $1$ are clearly inside the unit circle, resulting in decay towards a constant.}
\label{fig:upwind}
\end{figure}


Our second example considers a method where the spatial derivative is approximated by standard second order central finite differences ($u_x \approx \frac{u(x+h)-u(x-h)}{2h}$), thus the corresponding method of lines matrix has purely imaginary eigenvalues. To make the scheme stable, we need to time-step via an ODE solver whose stability region contains at least some portion of the imaginary axis. We choose the Shu-Osher method \cite{ShuOsher1988} (a SSP scheme \cite{GottliebShuTadmor2001}), whose intercept on the imaginary axis renders the scheme conditionally stable. The eigenvalues of $M_{h}$ (shown in \cref{fig:SSP_time_central_space}) again have only $\lambda_1 = 1$ on the unit circle; all other modes decay to zero as $n\to\infty$. So the $n\to\infty$ limit is the same as with upwind. Note that for this scheme (and even more so for methods of higher order), some modes decay slowly (and slower and slower as $h\to 0$). However, the limits $h\to 0$ and $\limsup_{n\to\infty}$ nevertheless do not commute.

\begin{figure}[ht]
\begin{minipage}[b]{.37\textwidth}
\includegraphics[width=\textwidth]{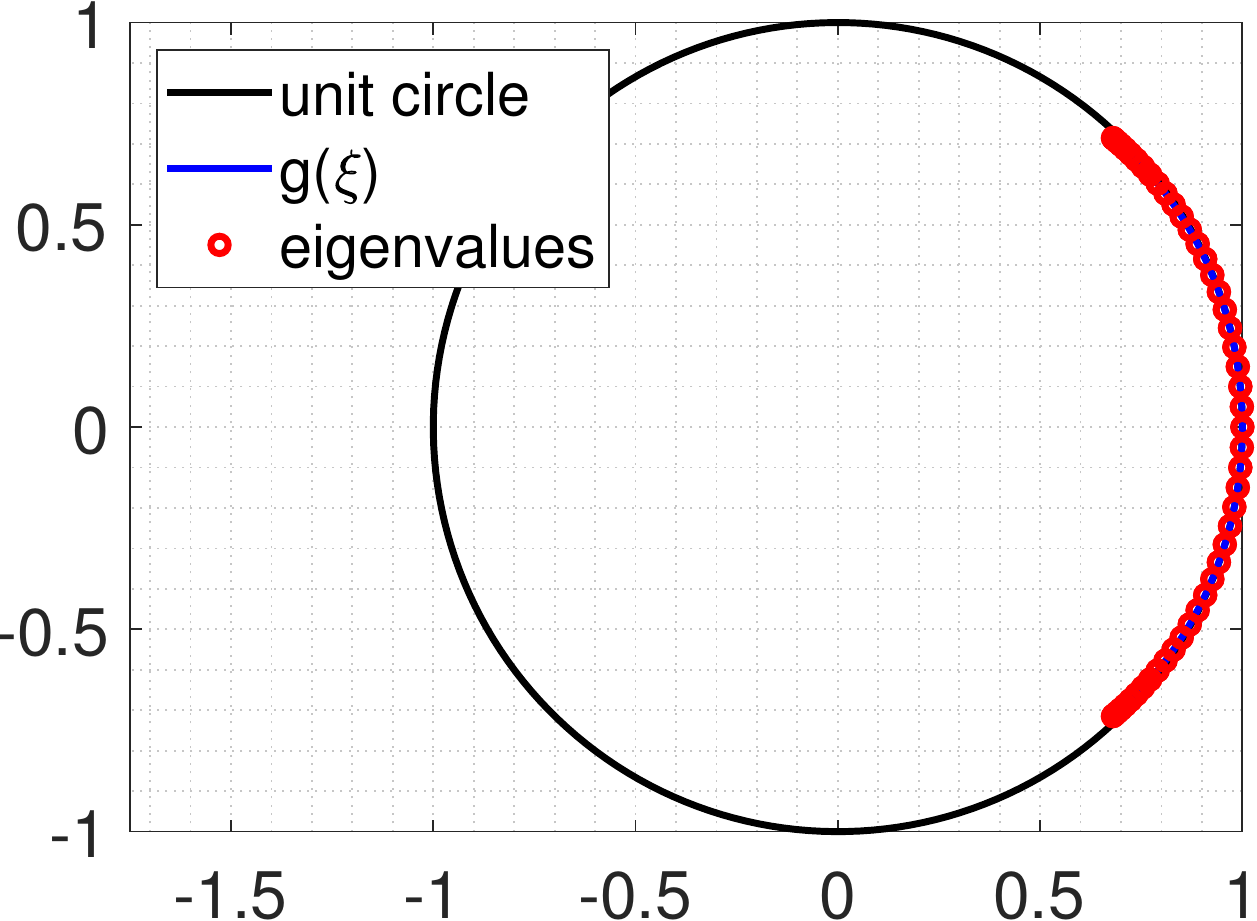}
\end{minipage}
\hfill
\begin{minipage}[b]{.57\textwidth}
\includegraphics[width=\textwidth]{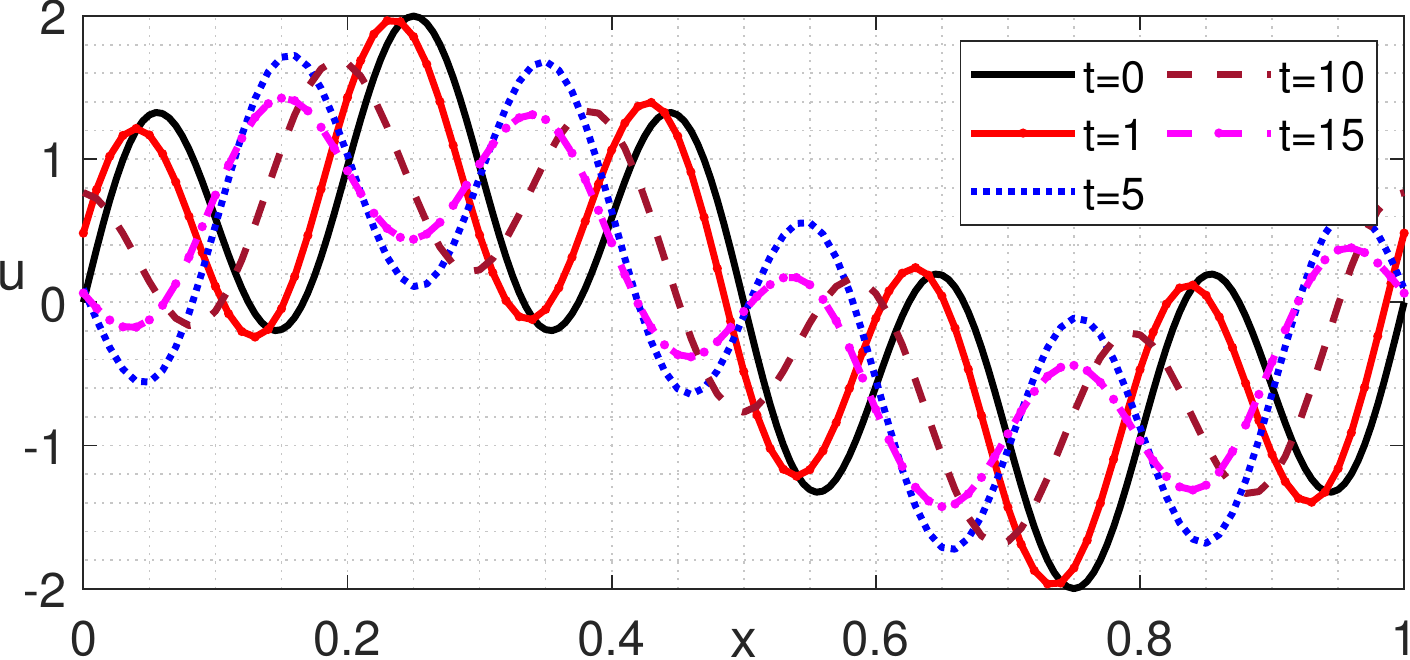}
\end{minipage}
\caption{SSP-RK$3$ in time and central in space method for $u_t+u_x=0$ with $\mu=0.8 $ and $m=100$ grid points: Amplification factor $g(\xi)=\left(1-\frac{\mu^2}{4}\right)+\frac{\mu^2}{4} \cos{2 \xi h}- i \left[ \frac{\mu^3}{24}\sin{3 \xi h}+\left( \mu-\frac{\mu^3}{8}\right)\sin{\xi h} \right]$, and eigenvalues (left) and numerical solution at different times (right). Despite high order, all the eigenvalues except $1$ are inside the unit circle, resulting in decay towards a constant.}
\label{fig:SSP_time_central_space}
\end{figure}


The last example is a method where the spatial derivative is approximated by the fourth-order central finite difference $u_x \approx \frac{-u(x+2 h)+8 u(x+h)-8 u(x-h)+  u(x-2 h)}{12 h}$, leading to a method of lines matrix with all purely imaginary eigenvalues. Now, we time-step via Crank-Nicolson (which we use for the sake of argument --- implicit methods are non-popular for advection problems \cite{LeVeque2007}). Because the imaginary axis is exactly the stability boundary for Crank-Nicolson, the resulting scheme is neutrally stable. In contrast to the other examples above, the third scheme shown in \cref{fig:CN_time_4th_space} has all eigenvalues on the unit circle, i.e., the scheme is purely dispersive. Consequently, all modes contained in the initial ($t=0$) approximation persist for all time. However, they move at different speeds, and generally do not reproduce the initial conditions after integer multiples of $1/a$, a fact that holds true broadly, see below.

\begin{figure}[ht]
\begin{minipage}[b]{.37\textwidth}
\includegraphics[width=\textwidth]{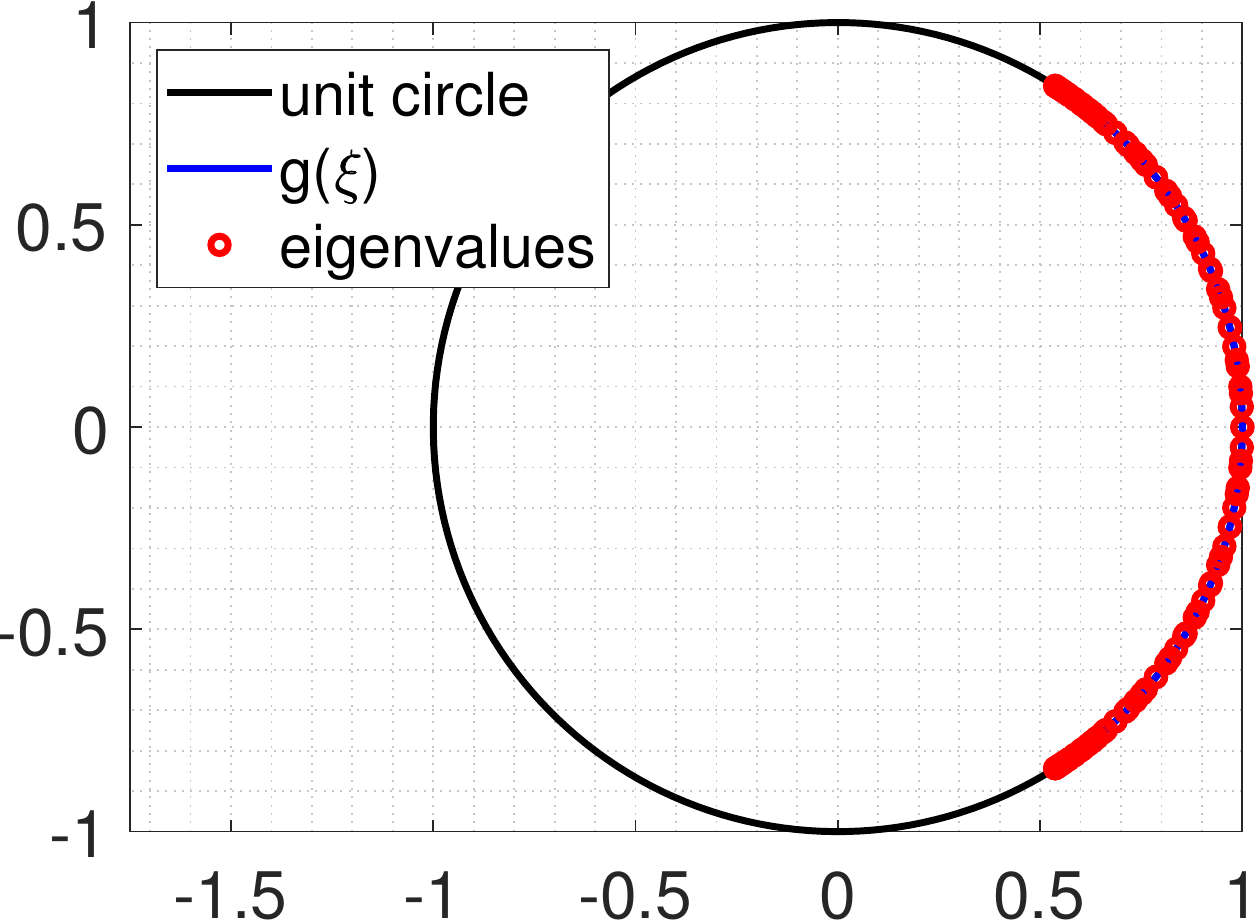}
\end{minipage}
\hfill
\begin{minipage}[b]{.57\textwidth}
\includegraphics[width=\textwidth]{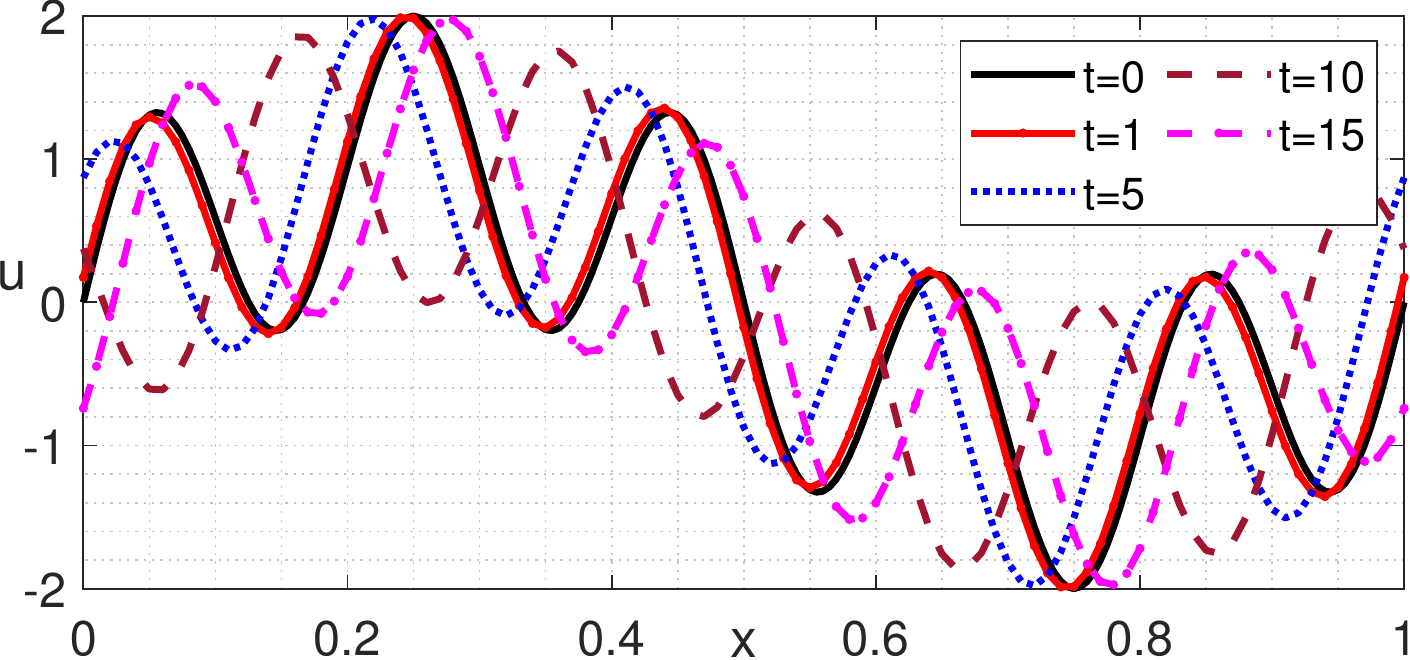}
\end{minipage}
\caption{Crank-Nicolson in time and fourth-order in space method for $u_t+u_x=0$ with $\mu=0.8 \ (\mathrm{CFL})$ and $m=100$ grid points: Amplification factor $g(\xi)=\frac{1-i\left( \frac{2 \mu}{3}\sin{\xi h}-\frac{\mu}{12}\sin{2 \xi h} \right)}{1+i\left( \frac{2 \mu}{3}\sin{\xi h}-\frac{\mu}{12}\sin{2 \xi h} \right)}$, and eigenvalues (left) and numerical solution at different times (right). All eigenvalues are on the unit circle, hence no modes decay; however, they move at wrong speeds.}
\label{fig:CN_time_4th_space}
\end{figure}


\subsection{Qualitative analysis}
\label{subsec:qualitative_analysis}
The numerical results above highlight some principles how linear schemes tend to fail to yield commuting limits. To prove that (generic) linear methods in fact do not have commuting limits property \cref{commuting_limits}, we consider problem \cref{eq:advection_1d} on the whole real line and assume the following about a linear scheme:
\begin{assumption}\label{assumptions_on_lin_methods}
~
    \begin{enumerate}[(i)]
        \item Consider a sequence of schemes (parameterized by $h$) with fixed CFL number $\mu = \frac{a \Delta t}{h}$. Moreover, assume $0<\mu<1$ to exclude trivial schemes.
        \label{assumptions_on_lin_methods_5}
        \item The scheme is defined with constant coefficients $c_{j}^{h}$ for each grid point. 
        \label{assumptions_on_lin_methods_1}
        \item The scheme uses same update stencil for each grid point at every time step, and (for fixed $\mu$) the stencil is independent of $h$, i.e., $c_{j}^{h}=c_j \; \forall \; h$. \label{assumptions_on_lin_methods_2}
        \item The scheme is stable and consistent.\label{assumptions_on_lin_methods_3}
        \item The initial condition $u_{0} \in C(\mathbb{R}) \cap {L}^2(\mathbb{R})$. \label{assumptions_on_lin_methods_4}
    \end{enumerate}
\end{assumption}

We employ Von Neumann stability analysis to characterize the stability assumption of the scheme. The analysis is particularly simple for a Cauchy problem \cite{LeVeque2007}, i.e., we consider \cref{eq:advection_1d} on the real line $-\infty < x < \infty$ with no boundaries. We discretize the whole domain and approximate the initial condition to get the grid function $\boldsymbol{u}^{h,0}=\left\{ u_{j}^{0} \right\}_{j=-\infty}^{\infty}$, where $u_{j}^{0}$ is the evaluation of initial condition at grid points $x_j=j h$ for $j \in \mathbb{Z}$. As above, here we discard the superscript $h$ in the components of $\boldsymbol{u}^{h,0}$, for notational economy. As we consider a sequence of meshes and approximate the initial condition on infinite grids $h\mathbb{Z}$, here we appropriately use the scaled $2$-norm, $\|\boldsymbol{u}^{h,0}\|_{{\ell}^2}:=\left(h \sum_{j=-\infty}^{\infty}\left|u_{j}^{0}\right|^{2}\right)^{1 / 2}$ that is consistent with the notation in \cite{LeVeque2007}. Then from \cref{assumptions_on_lin_methods_4}, it follows that $\boldsymbol{u}^{h,0} \in {\ell}^2(h \mathbb{Z})$. Now we can express $u_{j}^{0}$ by using the semidiscrete Fourier transform \cite{trefethen2000spectral}, defined by the map
\begin{align*}
    \mathcal{F} \colon \left({\ell}^2(h \mathbb{Z}),\|.\|_{{\ell}^2} \right) &\to \left({{L}}^2\left[-\frac{\pi}{h},\frac{\pi}{h}\right],\|.\|_{{{L}}^2} \right) \\
    \boldsymbol{u}^{h,0} &\mapsto \mathcal{F}(\boldsymbol{u}^{h,0})=\hat{u}(\xi)={h} \sum_{j=-\infty}^{\infty} u_{j}^{0} e^{-i j h \xi}\;,
\end{align*}
and its inverse map
\begin{align*}
    \mathcal{F}^{-1} \colon \left({{L}}^2\left[-\frac{\pi}{h},\frac{\pi}{h}\right],\|.\|_{{{L}}^2} \right) &\to  \left({\ell}^2(h \mathbb{Z}),\|.\|_{{\ell}^2} \right) \\
    \hat{u} &\mapsto \mathcal{F}^{-1}(\hat{u})=\boldsymbol{u}^{h,0}= \left\{ u_{j}^{0} \right\}_{j=-\infty}^{\infty}\;,
\end{align*}
where $u_{j}^{0} =\frac{1}{2 \pi}\int_{-\frac{\pi}{h}}^{\frac{\pi}{h}} \hat{u}(\xi) \ e^{i j h \xi} \textrm{d} \xi$. Note that our physical variable is discrete and unbounded, and corresponds to the space $\left({\ell}^2(h \mathbb{Z}),\|.\|_{{\ell}^2} \right)$, while the frequency variable is bounded and continuous, and refers to the space $\left({{L}}^2\left[-\frac{\pi}{h},\frac{\pi}{h}\right],\|.\|_{{{L}}^2} \right)$.
By \cref{assumptions_on_lin_methods_1} and \cref{assumptions_on_lin_methods_2}, the one-step update rule for any linear scheme using the stencil $[x_j-s h,x_j-(s-1) h,\ldots, x_j+s h]$ can be written as
\begin{equation}
\label{lin_method_update_rule}
    U_{j}^{n+1} = \sum_{\nu=-s}^{s} c_{\nu}U_{j+\nu}^{n}\;.
\end{equation}
As the initial condition $u_{0} \in C(\mathbb{R}) \cap {L}^2(\mathbb{R})$, the approximated initial grid function $\boldsymbol{u}^{h,0}$ is in ${\ell}^2(h \mathbb{Z})$. We decompose the initial grid function into Fourier modes by using the semidiscrete Fourier transform and express the grid values as $u_{j}^{0}  =\frac{1}{2 \pi}\int_{-\frac{\pi}{h}}^{\frac{\pi}{h}} \hat{u}(\xi) \ e^{i j h \xi} \textrm{d} \xi, \ \forall \ j.$
We now look at the effect of the numerical method on each of the Fourier modes and combine all the modified modes to obtain the numerical solution at some given time. Using equation \cref{recurrence_relation} and the numerical scheme \cref{lin_method_update_rule} on the grid function $U_j^{n}=e^{i \xi j h}$ of single wave number $\xi$, we calculate $g(\xi)=\sum_{\nu=-s}^{s} c_{\nu}e^{i \nu h \xi}$. Using the recurrence relation \cref{recurrence_relation} on each mode and combining them all we obtain the numerical solution $\boldsymbol{U}^{h,t_n}$ at $t_n$ given by its components
\begin{equation*}
    U_{j}^{n} = \frac{1}{2 \pi}\int_{-\frac{\pi}{h}}^{\frac{\pi}{h}}\hat{u}(\xi)[g(\xi)]^n e^{i \xi j h}\textrm{d} \xi, \ \forall \ j\;.
\end{equation*}
We seek to compare the numerical solution with the true solution in continuum $L^2$ norm (as we employ the Fourier modes), so it is natural to consider the trigonometric interpolant based on the grid function $\left\{ U_{j}^{n}\right\}_{j=-\infty}^{\infty}$ that is defined by
\begin{equation}
\label{eq:numerical_solution_trig_interp}
    \mathcal{I}(\boldsymbol{U}^{h,t_n})(x) := \frac{1}{2 \pi}\int_{-\frac{\pi}{h}}^{\frac{\pi}{h}}\hat{u}(\xi)[g(\xi)]^n e^{i \xi x}\textrm{d} \xi, \ \forall \ x \in \mathbb{R}\;,
\end{equation}
as our numerical solution function (defined for all $x$) at $t_n$.
A natural (yet not the only possible) way to initialize the numerical scheme is with the discrete data $\left\{ u_{j}^{0} \right\}_{j=-\infty}^{\infty}$, directly recovered from the initial profile. The numerical schemes now evolve this discrete solution over time, and the interpolant \eqref{eq:numerical_solution_trig_interp} defines a corresponding continuum function.

Here we are only interested in studying the schemes' evolution error, so we employ an error measure that yields zero error at time $t=0$: rather than comparing the numerical approximation \eqref{eq:numerical_solution_trig_interp} to the true initial profile $u_0$, we instead compare it to a \emph{proxy} of the true solution, the trigonometric interpolant of the initial grid function, $\mathcal{I}(\boldsymbol{u}^{h,0})(x)  =\frac{1}{2 \pi}\int_{-\frac{\pi}{h}}^{\frac{\pi}{h}} \hat{u}(\xi) \ e^{i \xi x} \textrm{d} \xi, \ \forall \,  x \in \mathbb{R}$. Moreover, like the true solution is just the shift of the initial profile, so is the proxy of the true solution at time $t_n$ obtained via shifting the function $\mathcal{I}(\boldsymbol{u}^{h,0})(x)$ by $n \Delta t$ to yield
\begin{equation}
\label{true_solution}
     \mathcal{I}(\boldsymbol{u}^{h,t_n})(x):=\frac{1}{2 \pi}\int_{-\frac{\pi}{h}}^{\frac{\pi}{h}} \hat{u}(\xi) e^{i (x-n \Delta t) \xi} \,\textrm{d} \xi, \ \forall \ x  \in \mathbb{R}\;.
\end{equation}
Having a stable scheme requires that $|g(\xi)|\leq 1$. All components corresponding to $\xi$ with $|g(\xi)|< 1$ will decay to zero as $n\to\infty$, and all that remains are the components corresponding to $\xi$ with $|g(\xi)|= 1$. Let $\mathcal{C} = \{\xi : |g(\xi)| = 1\}$, and for all $\xi \in \mathcal{C}$ write $g(\xi) = e^{\mathrm{i}\theta(\xi)}$.
This means that the exact numerical solutions of a linear scheme are precisely the Fourier modes that have their amplification on the unit circle and move with the correct speed as the true solution. For a linear scheme to have the commuting limits property \cref{commuting_limits}, we would need $\mathcal{C}$ to possess increasingly many elements that move with the correct speed, as $h\to 0$ (and thus the number of grid points goes to infinity). However, all linear schemes satisfying \cref{assumptions_on_lin_methods} have the property that as $h\to 0$, the Fourier modes approaches an algebraic curve in the complex plane. Consequently, one of two scenarios arises: either $\mathcal{C}$ has only finitely many elements, or all Fourier modes lie on unit circle. In the latter case, the scheme is purely dispersive; however, expect for trivial cases, the wave speeds of the different Fourier modes do not match up in a way to recover an exact solution. Regarding this result, we now present the following technical lemma, which we are going to use to prove our goal.
\begin{lemma}
\label{lemma1}
     Under \cref{assumptions_on_lin_methods}, one has
     \begin{equation*}
         \limsup_{n \to \infty}\left |[g(\xi)]^n- e^{-i \xi n \mu h}\right | \geq C\neq 0\;,
     \end{equation*}
    where $\xi$ is a wave number, $\mu$ is the $\mathrm{CFL}$ number, and $C$ is a positive constant.
\end{lemma}
\begin{proof}
    The amplification factor is $g(\xi)=\sum_{\nu=-s}^{s} c_{\nu}e^{i \nu h \xi}$. By a change of variable $\omega=\xi h$, $g$ becomes a function of $\omega$, and we denote this function by
    $\tilde{g}(\omega)=\sum_{\nu=-s}^{s} c_{\nu}e^{i \nu \omega}, \ \omega \in [-\pi,\pi].$
    The $\limsup_{n \to \infty}\left |[g(\xi)]^n- e^{-i \xi n \mu h}\right |$ becomes $\limsup_{n \to \infty}\left |[\tilde{g}(\omega)]^n- e^{-i \omega n \mu }\right |$.
    Now, if $|\tilde{g}(\omega)|<1$ for $\omega  \in [{-\pi},{\pi}]$, then
    \begin{equation*}
    \limsup_{n \to \infty}\left |[\tilde{g}(\omega)]^n- e^{-i \omega n \mu }\right |=1\;.
    \end{equation*}
    For $\omega$'s with $|\tilde{g}(\omega)|=1$ we can write $\tilde{g}(\omega)=e^{i \theta(\omega)}$, where $\theta(\omega) \in [0,2\pi)$ and then
    \begin{align*}
        \limsup_{n \to \infty}\left |[\tilde{g}(\omega)]^n- e^{-i \omega n \mu }\right | & = \limsup_{n \to \infty}\left |e^{i n \theta(\omega) }- e^{-i \omega n \mu }\right | = \limsup_{n \to \infty}\left |e^{i[ \theta(\omega) + \omega \mu]n }- 1 \right | \\
        & =\limsup_{n \to \infty}\left |z^n(\omega)- 1 \right |, \ \textrm{where} \ z(\omega)=e^{i[ \theta(\omega) + \omega \mu]}\;.
    \end{align*}
    Clearly, $|z(\omega)|=1$ for all $\omega \in [-\pi,\pi]$, and $z(0)=1$ as $g(0)=1$ follows from the consistency (\cref{assumptions_on_lin_methods_3}) of the scheme. We claim that $z(\omega)=1$ only for finitely many $\omega$'s in $[-\pi,\pi]$. If not, then it follows that $\tilde{g}(\omega)  = e^{-i \omega \mu}$ on a subset $D$ of $[-\pi,\pi]$ with a limit point. As $\tilde{g}(\omega)$ and $e^{-i \omega \mu}$ are analytic functions, we have $\tilde{g}(\omega)  = e^{-i \omega \mu}$ on the whole domain $[-\pi,\pi]$ (a property of complex analytic functions), which implies $\sum_{\nu=-s}^{s} c_{\nu}e^{i \nu \omega} = e^{-i \omega \mu}, \ \forall \, \omega \in [-\pi,\pi]$. This is a contradiction to \cref{assumptions_on_lin_methods_5} because $\{e^{i \nu \omega}:\nu=-s,\dots,s\}$ is a linearly independent set, and the only way this identity is true is when $\mu=\nu$, $c_{-\nu}=\delta_{\nu \mu}$ and $|\mu|\leq s$. But this will lead to a trivial numerical scheme. So we conclude that $z(\omega) \neq 1$ almost everywhere in $[-\pi,\pi]$.
Then for almost all $\omega \in [-\pi,\pi]$, $\limsup_{n \to \infty}\left |[\tilde{g}(\omega)]^n- e^{-i \omega n \mu }\right | =\limsup_{n \to \infty}\left |[z^n(\omega)- 1 \right | \geq \underbrace{\left |e^{i\frac{2 \pi}{3} }- 1 \right |}_{=:C} \;.$
    \begin{figure}[htb]
    \vspace{-1em}
    \begin{center}
        \begin{tikzpicture}[scale=.50]
        \draw[black, thick] (0,0) rectangle (5,5);
        \draw[blue, thick](2.5,2.5) circle (1.5);
        \draw[gray, thick] (2.5,0) -- (2.5,5);
        \draw[gray, thick] (0,2.5) -- (5,2.5);
        \draw[gray, thick] (4,2.5) -- ({1.5*cos(120)+2.5},{1.5*sin(120)+2.5});
        \node[below =2pt of {(3,3.9)}] {c};
        \draw[gray, thick] (2.5,2.5) -- ({1.5*cos(120)+2.5},{1.5*sin(120)+2.5});
         \node[below =2pt of {(4.5,4.9)}] {$\mathbb{C}$};
        \filldraw [black] ({1.5*cos(120)+2.5},{1.5*sin(120)+2.5}) circle (2pt) node[above]{$z$};
        \filldraw [black] ({1.5*cos(240)+2.5},{1.5*sin(240)+2.5}) circle (2pt)
        node[below]{$z^2$};
        \filldraw [black] ({1.5*cos(360)+2.5},{1.5*sin(360)+2.5}) circle (2pt)
        node[below right]{$z^3$};
        \coordinate (O) at (2.5,2.5);
        \coordinate (A) at (4,2.5);
        \coordinate (B) at ({1.5*cos(120)+2.5},{1.5*sin(120)+2.5});
        \draw (2.5,2.5) -- (2.8,2.5) arc (0:120:0.3cm)-- cycle;
        \node[below right=2pt of {(2.5,2.5)}] {$\frac{2 \pi}{3}$};
        \end{tikzpicture}
        \caption{Geometry for bound on the estimate.}
        \label{fig:lower_bound}
    \end{center}
    \end{figure}
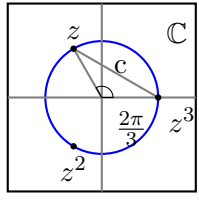
The last inequality follows by considering all possible values of the argument of the complex number $z=e^{i[ \theta(\omega) + \omega \mu]} \neq 1$ that is taken to arbitrary powers $n$, see \cref{fig:lower_bound}. Hence, in any case
\begin{align*}
    \limsup_{n \to \infty}\left |[g(\xi)]^n- e^{-i \xi n \mu h}\right |\geq C \neq 0, \ \text{independent of} \ h.
\end{align*}
\end{proof}

\begin{theorem}
    A linear scheme for \cref{eq:advection_1d} under \cref{assumptions_on_lin_methods} cannot have the commuting limits property, i,e.,
    \begin{equation*}
    \lim_{h\to 0}\limsup_{n\to\infty} \boldsymbol{\varepsilon}^{h,t_n} \neq 0\;.
    \end{equation*}
\end{theorem}
\begin{proof}
    For a given $h$ and $t_n$, we measure the error in the continuum $\|\cdot\|_{L^2}$ norm,  $ \boldsymbol{\varepsilon}^{h,t_n} = \|\mathcal{I}(\boldsymbol{U}^{h,t_n})-\mathcal{I}(\boldsymbol{u}^{h,t_n})\|_{L^2}.$
     Using \cref{eq:numerical_solution_trig_interp}, \cref{true_solution} and the fact that $\Delta t = \mu h$, we get $    \big(\mathcal{I}(\boldsymbol{U}^{h,t_n})-\mathcal{I}(\boldsymbol{u}^{h,t_n})\big)(x)=\frac{1}{2 \pi}\int_{-\frac{\pi}{h}}^{\frac{\pi}{h}}\hat{u}(\xi)\left([g(\xi)]^n- e^{-i \xi n \mu h} \right) e^{i \xi x}\textrm{d} \xi, \ \forall \ x  \in \mathbb{R}\;.$
    Let
    \begin{equation*}
    d(\xi) =
    \begin{cases}
    \hat{u}(\xi)\left([g(\xi)]^n- e^{-i \xi n \mu h} \right) & \text{for~} x \in [-\frac{\pi}{h},\frac{\pi}{h}]\\
    0 & \text{otherwise} \;,
    \end{cases}
    \end{equation*}
     then clearly $d(\xi) \in {{L}}^2(\mathbb{R}) $, as $\hat{u} \in {{L}}^2\left[-\frac{\pi}{h},\frac{\pi}{h}\right]$ and $|g(\xi)| \leq 1$, by \cref{assumptions_on_lin_methods_3}. Also by construction, $d(\xi)$ is the Fourier transform of the function $\big(\mathcal{I}(\boldsymbol{U}^{h,t_n})-\mathcal{I}(\boldsymbol{u}^{h,t_n})\big)(x)$, for $x \in \mathbb{R}$, following from \cite{trefethen2000spectral}. Using Parseval's identity we can write
    \begin{equation*}
         \boldsymbol{\varepsilon}^{h,t_n} = \|\big(\mathcal{I}(\boldsymbol{U}^{h,t_n})-\mathcal{I}(\boldsymbol{u}^{h,t_n})\big)(x)\|_{L^2} = \|d(\xi)\|_{{{L}}^2}\;.
    \end{equation*}
    First taking $\limsup_{n \to \infty}$ and then taking the limit as $h \to 0$ of the error $\boldsymbol{\varepsilon}^{h,t_n}$, we get
    \begin{align*}
         \lim_{h\to 0}\limsup_{n\to\infty} \boldsymbol{\varepsilon}^{h,t_n} & = \lim_{h\to 0}\limsup_{n\to\infty} \|d(\xi)\|_{{{L}}^2} \\
         & = \lim_{h\to 0}\limsup_{n\to\infty} \left( \int_{-\frac{\pi}{h}}^{\frac{\pi}{h}} \left| \hat{u}(\xi)\left([g(\xi)]^n- e^{-i \xi n \mu h} \right)\right|^{2} \textrm{d} \xi \right)^{\frac{1}{2}} \\
         & = \lim_{h\to 0}\left( \int_{-\frac{\pi}{h}}^{\frac{\pi}{h}} \left(\left| \hat{u}(\xi) \right| \limsup_{n\to\infty}  \left| [g(\xi)]^n- e^{-i \xi n \mu h} \right|\right)^{2} \textrm{d} \xi \right)^{\frac{1}{2}} \\
         & \overset{\text{lemma~}\ref{lemma1}}\geq C \lim_{h\to 0}  \left( \int_{-\frac{\pi}{h}}^{\frac{\pi}{h}}\left| \hat{u}(\xi) \right|  \textrm{d} \xi \right)^{\frac{1}{2}} = \ C \|\hat{u}\|_{{L}^2}>0\;.
    \end{align*}
\end{proof}

\subsection{Quantitative scaling}
Having shown that linear schemes are not suitable to produce accurate solutions for really long times, we now provide scaling laws of the errors incurred by these schemes depending on the grid size $h$ and the final time $t_\textrm{f}$. In order to have a chance to obtain an accurate infinite time limit solution, we require that, for a fixed h, the numerical schemes produce non-constant numerical approximations as $n \to \infty$. In practice, we cannot take arbitrarily small $\Delta t$ or $h$, but we are happy to accept $h$ for which the numerical scheme will produce the desired error to exhibit its correct order of convergence. For a given final time $t_\mathrm{f}$, there exists a critical value of $h=h^{*}$, for which the method will produce that desired error. Now, if we take a larger final time $t_\mathrm{f}$, then it will be reasonable to believe that we would require a smaller value of $h^{*}$ to produce the desired error. So, the critical value of $h$ depends on the choice of final time $t_\mathrm{f}$, and it decreases as $t_\mathrm{f}$ increases. To rigorously understand this fact we consider \eqref{eq:advection_1d} and solve this problem by second-order accurate standard Lax-Wendroff method \cite{LeVeque2007}. The modified equation \cite{LeVeque2007} for this numerical method is given by
\begin{equation}\label{eq:LW modified Eq}
v_t+a v_x=-\frac{1}{6} a h^2 \left( 1- { \left(\frac{a \Delta t}{h}\right) }^2 \right) v_{xxx}.
\end{equation}
The true solution of \eqref{eq:LW modified Eq} captures the behavior of the numerical solutions produced by the Lax-Wendroff method well. As $\frac{a \Delta t}{h}=\mu$ is a fixed number (CFL), we can write this equation as $ v_t+a v_x=Ch^2 v_{xxx},$ where $C=-\frac{a}{6} ( 1-{\mu}^2 )$. The differential operators $\frac{\partial}{\partial x}$ and $\frac{\partial^3}{\partial x^3}$ commute, which implies that we can solve the equation by splitting into advection and dispersion problems in any order. To understand the error made by the scheme, we can ignore the advection part as it is also present in the original problem \eqref{eq:advection_1d}. So, it suffices to consider only $v_t=Ch^2 v_{xxx}$  to understand the change in the numerical solution over time. We compare the solution of this equation, denoted by $w_h(x,t)$, with the solution of the original problem with zero velocity, i.e., $u(x,t)=u_{0}(x)$. By the time-scaling $t \mapsto h^2t$, we can deduce that $w_{h}(x,t)=w_{1}(x,h^2t)$. We measure the error for a given final time $t_\mathrm{f}$ and the corresponding critical resolution $h^{*}$, by $\boldsymbol{\varepsilon}^{h^{*},t_\mathrm{f}}=||w_1(\cdot,(h^{*})^2 t_\mathrm{f})-u_0(\cdot)||_{{L}^1}$ in the continuum $L^1$ norm. We want this error to be a constant which is the desired error to get second order accuracy for the Lax-Wendroff method. The desired error depends on $h^{*}$ and $t_\mathrm{f}$, and this forces $(h^{*})^2 t_\mathrm{f}=C_1,$ for some constant $C_1$. In fact, for a $p$-th order linear scheme, self-similarity arguments based on the modified equation reveal that the error generally scales like $\boldsymbol{\varepsilon}^{h,t_\mathrm{f}} \propto h^p t_\mathrm{f}$, i.e., doubling $t_\mathrm{f}$ also doubles the error, and $h^* \propto t_\mathrm{f}^{-1/p}$. Even though for a fixed $t_\mathrm{f}$ the error decreases as $h\to 0$, for a fixed $h$ the error grows as $t_\textrm{f} \to\infty$, thus preventing the commuting limits property \eqref{commuting_limits}. The typical shape of the bivariate error function $\boldsymbol{\varepsilon}^{h,t_\mathrm{f}}$ is shown in \cref{fig:error_convergence} (left panel), here for the Lax-Wendroff method. For any fixed $t_\mathrm{f}$, the error $\boldsymbol{\varepsilon}^{h,t_\mathrm{f}}$ is L-shaped, with a transition mesh size $h^*(t_\mathrm{f})$ below which a clean second order manifests.

\begin{figure}[ht]
\begin{minipage}[b]{.49\textwidth}
\includegraphics[width=\textwidth]{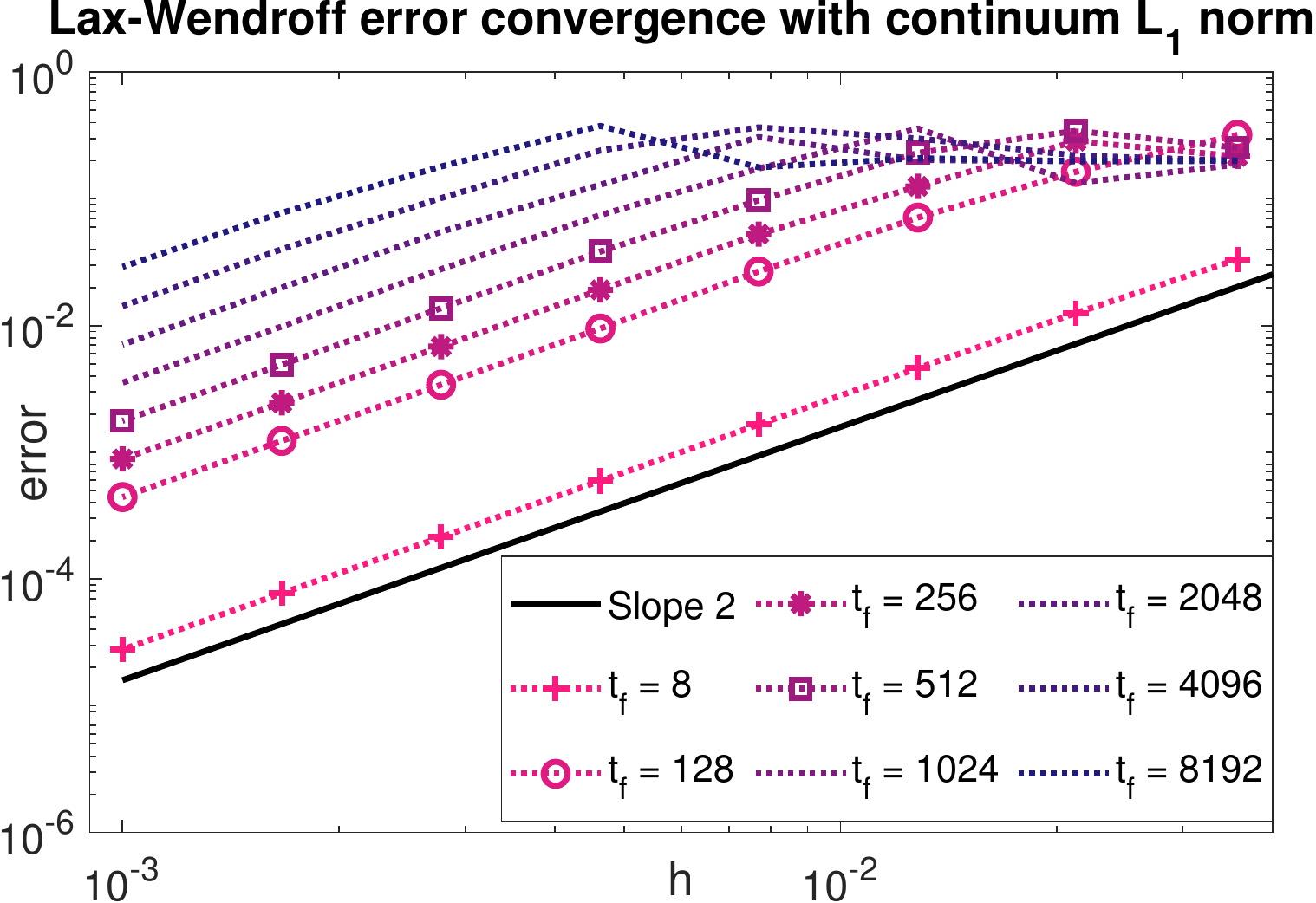}
\end{minipage}
\hfill
\begin{minipage}[b]{.49\textwidth}
\includegraphics[width=\textwidth]{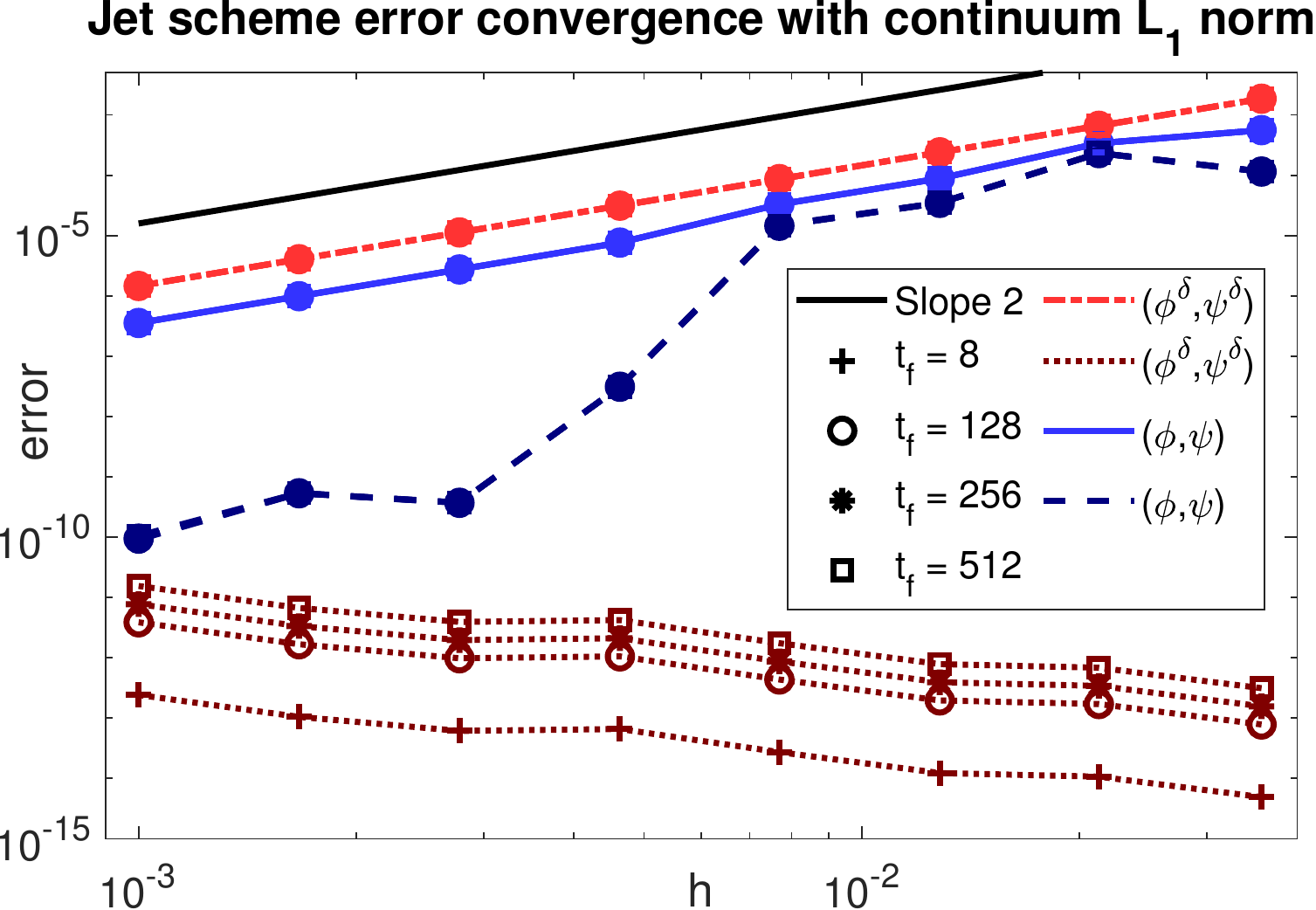}
\end{minipage}
\caption{Approximation error $\varepsilon^{h,t_\mathrm{f}}$ as function of $h$ and $t_\mathrm{f}$. \quad Left: For the linear (second order) scheme, $\varepsilon^{h,t_\mathrm{f}}$ decreases with $h$ but steadily grows in $t_\mathrm{f}$; hence there is no $h\to 0$ convergence that is uniform in $t_\mathrm{f}$. \quad Right: For the nonlinear jet scheme, the light blue/red curves show the total error at different final times for two approximations of a smooth initial function: (i)~direct evaluation of $u_0(jh)$ and $u_0'(jh)$, denoted $(\phi,\psi)$; and (ii)~the special strategy in \cref{modified_IC}, denoted $(\phi^{\delta},\psi^{\delta})$. The corresponding errors of only the jet scheme evolution (without i.c.\ approximation errors) are shown by the dark blue/red curves for (i)/(ii), respectively. For the jet scheme, the total error $\varepsilon^{h,t_\mathrm{f}}$ decreases with $h$ and it is (essentially) constant w.r.t.~$t_\mathrm{f}$. Hence the scheme has the potential to be exact. Even though there is no traditional convergence rates due to the structure of fixed points (see \cref{sec:error convergence} for a more detailed discussion of the convergence of jet schemes), the jet scheme's evolution error (dark blue) is observed to decrease (in $h$) more rapidly than the approximation error of the i.c.}
\label{fig:error_convergence}
\end{figure}

\vspace{1.5em}
\section{Jet scheme with nonlinear interpolation}
\label{sec:jet scheme}
We now demonstrate that there \emph{are} in fact numerical methods that are exact. To that end, we use the jet scheme formalism, established in \cite{NaveRosalesSeibold2010, SeiboldRosalesNave2012, ChidyagwaiNaveRosalesSeibold2012} and extended, e.g., in \cite{Salac2011}. Jet schemes are semi-Lagrangian approaches that track and evolve derivative information in addition to function values. In the original papers, the derivative information enables one to achieve high order, while maintaining optimally local update rules, based on Hermite interpolation. This concept can also be extended to higher derivatives and higher space dimensions \cite{SeiboldRosalesNave2012}. The jet scheme methodology is as follows.
\begin{method}\label{numerical_method}
~
    \begin{enumerate}[1.]
        \item Given function value ($\phi$) and derivative ($\psi$) approximations at all grid points at time $t$, to obtain the state at time $t+\Delta t$, first, for each grid point, solve the characteristic equation $\dot{x} = a(x,t)$ backwards from $t+\Delta t$ to $t$ to find the associated foot point $x_\mathrm{ft}$.
        \item Evaluate (based on the grid data) the Hermite interpolant $P(x_\mathrm{ft})$ and its derivative $P'(x_\mathrm{ft})$ at this foot point.
        \item Then, using these evaluations, evolve the characteristic equations $\dot{\phi} = 0$ and $\dot{\psi} = -a_x(x,t)\psi$ (if $a_x=0$, this reduces to $\dot{\psi} = 0$) forward from $t$ to $t+\Delta t$ to yield the new function value and derivative information at the grid points.
    \end{enumerate}
    Suitable approximations of the initial condition are discussed later in this section.
\end{method}

While the jet scheme methodology works in principle with any Hermite interpolant, it has thus far been established only with interpolations that are linear operators (including in the context of the gradient-augmented level set method \cite{NaveRosalesSeibold2010}).
In contrast, here we propose a specific nonlinear interpolant that is tailored to produce the exactness properties requested in \cref{sec:convergence_and_the_problem}. Thus the method is nonlinear as the update rule depends on the state vector nonlinearly.
The nonlinear jet scheme interpolant is defined as follows (in 1D). Given data $\phi(x_\text{L})$, $\phi(x_\text{R})$ (function values) and $\psi(x_\text{L})$, $\psi(x_\text{R})$ (derivatives) we wish to define an interpolant $P(x)$ on $[x_\text{L},x_\text{R}]$. We consider the two linear functions emanating from $x_\text{L}$ and $x_\text{R}$:
\begin{equation*}
L_{\text{L}}(x) := \phi(x_\text{L})+\psi(x_\text{L})(x-x_\text{L})
\quad\text{and}\quad
L_{\text{R}}(x) := \phi(x_\text{R})+\psi(x_\text{R})(x-x_\text{R})\;.
\end{equation*}
If $\psi(x_\text{L}) \neq \psi(x_\text{R})$ and the lines' intersection point $x_\text{k} = \frac{\phi(x_\text{L})-\phi(x_\text{R})-\psi(x_\text{L})x_\text{L}+\psi(x_\text{R})x_\text{R}}{\psi(x_\text{R})-\psi(x_\text{L})}$ lies inside $(x_\text{L},x_\text{R})$, we use the piecewise linear function formed by them (see \cref{fig:nonlinear_interpolant_genuine}):\vspace{-1em}
%
\begin{figure}[htb]
\begin{center}
\begin{minipage}[b]{.35\linewidth}
\centering
\begin{tikzpicture}[scale=0.60]
\draw [thick](-3,0) -- (3,0);
\draw [thick](-2.5,0) -- (-2.5,3);
\draw [black!60!green,ultra thick,-,>=stealth] (-2.8,2.8) -- (-2.2,3.2);
\draw [black!60!green,ultra thick,-,>=stealth] (2.8,3.85) -- (2.2,4.15);
\draw [thick](2.5,0) -- (2.5,4);
\draw [red,thick,dashed] (-2.5,3) -- (0.5,5);
\draw [red,thick,dashed] (0.5,5) -- (2.5,4);
\draw [thick,dashed] (-0.8,0) -- (-0.8,4);
\draw [thick,dashed] (0.5,0) -- (0.5,5);
\filldraw [black] (-2.5,0) circle (2pt); 
\filldraw [black] (2.5,0) circle (2pt); 
\filldraw [red] (-0.8,0) circle (2pt); 
\filldraw [red] (0.5,0) circle (2pt); 
\node[below =2pt of {(-2.5,0)}] {$x_\text{L}$};
\node[below =2pt of {(-0.8,0)}] {$x_\mathrm{ft}$};
\node[below =2pt of {(0.5,0)}] {$x_\text{k}$};
\node[below =2pt of {(2.5,0)}] {$x_\text{R}$};
\filldraw [blue] (-2.5,3) circle (3pt);
\node[above =1pt of {(-1.67,2.0)}, text=blue] {$\phi(x_\text{L})$};
\node[above =1pt of {(-2.6,3.15)}, text=black!60!green] {$\psi(x_\text{L})$};
\node[above =1pt of {(2.4,4.1)}, text=black!60!green] {$\psi(x_\text{R})$};
\filldraw [blue] (-2.5,3) circle (3pt);
\filldraw [blue] (2.5,4) circle (3pt);
\node[above =1pt of {(1.65,3)}, text=blue] {$\phi(x_\text{R})$};
\end{tikzpicture}
\vspace{-.5em}
\label{fig:nonlinear_interpolant_genuine}
\caption{Genuine interpolant.}
\end{minipage}
\hspace{2em}
\begin{minipage}[b]{.35\linewidth}
\centering
\begin{tikzpicture}[scale=0.60]
\draw [thick](-3,0) -- (3,0);
\draw [thick](-2.5,0) -- (-2.5,3);
\draw [black!60!green,ultra thick,-,>=stealth] (-2.7,2.8) -- (-2.25,3.35);
\draw [black!60!green,ultra thick,-,>=stealth] (2.3,3.7) -- (2.7,4.3);
\draw [thick](2.5,0) -- (2.5,4);
\draw [red,thick,dashed] (-2.5,3) -- (2.5,4);
\draw [thick,dashed] (-0.8,0) -- (-0.8,3.3);
\filldraw [black] (-2.5,0) circle (2pt); 
\filldraw [black] (2.5,0) circle (2pt); 
\filldraw [red] (-0.8,0) circle (2pt); 
\node[below =2pt of {(-2.5,0)}] {$x_\text{L}$};
\node[below =2pt of {(-0.8,0)}] {$x_\mathrm{ft}$};
\node[below =2pt of {(2.5,0)}] {$x_\text{R}$};
\filldraw [blue] (-2.5,3) circle (3pt);
\node[above =1pt of {(-1.67,1.9)}, text=blue] {$\phi(x_\text{L})$};
\node[above =1pt of {(-2.5,3.1)}, text=black!60!green] {$\psi(x_\text{L})$};
\node[above =1pt of {(2.3,4.1)}, text=black!60!green] {$\psi(x_\text{R})$};
\filldraw [blue] (2.5,4) circle (3pt);
\node[above =1pt of {(1.65,2.6)}, text=blue] {$\phi(x_\text{R})$};
\end{tikzpicture}
\vspace{-.5em}
\caption{Fallback case.}
\label{fig:nonlinear_interpolant_fallback}
\end{minipage}
\end{center}
\end{figure}
\begin{equation*}
P(x) =
\begin{cases}
L_{\text{L}}(x)& \text{for~} x \le x_\text{k}\\
L_{\text{R}}(x)& \text{for~} x > x_\text{k} \;.
\end{cases}
\end{equation*}
Otherwise, we use the simple linear interpolant
\begin{equation*}
P(x) = \phi(x_\text{L})+\tfrac{\phi(x_\text{R})-\phi(x_\text{L})}{x_\text{R}-x_\text{L}}(x-x_\text{L})
\end{equation*}
as a fallback case (see \cref{fig:nonlinear_interpolant_fallback}). This last situation can be interpreted as having an extra line between the segments $L_{\text{L}}$ and $L_{\text{R}}$ whose length is taken to the maximum. Clearly, the fact that this new interpolant possesses jumps in its derivatives requires particular care in how evaluations of derivatives are defined. We discuss these aspects below.

\subsection{Two equivalent interpretations of jet schemes}
\label{subsec:two_jet_interpretations}
An important property of jet schemes is that they can be interpreted in two ways \cite{SeiboldRosalesNave2012}: (a)~an update rule in a finite-dimensional state space; or (b)~an update rule of functions. Here we establish that duality for the nonlinear scheme defined above.
For a given $m=1/h$, let $\boldsymbol{U}^{h,t_n} = [\phi_1^n,\psi_1^n,\phi_2^n,\psi_2^n,...,\phi_m^n,\psi_m^n]^T \in \mathbb{R}^{2m}$ denote the state vector that defines the jet scheme approximation at time step $n$. Here $\phi_j^n$ and $\psi_j^n$ approximate values and derivatives, respectively, of the true solution at the grid point $x_j$. Moreover, let $\mathcal{H}$ denote the space of all continuous piecewise linear functions on $[0,1]$ with finitely many kinks. With these spaces, \cref{numerical_method} can be written as the update rule
\begin{equation*}
\boldsymbol{U}^{h,t_{n+1}} = \tilde{\mathcal{N}} \boldsymbol{U}^{h,t_{n}}\;,\quad\text{where}\quad \tilde{\mathcal{N}} = \mathcal{E} \circ \mathcal{A} \circ \mathcal{I}\;.
\end{equation*}
Here $\mathcal{A}: \mathcal{H} \rightarrow \mathcal{H}$ is the true solution advection operator (which for problem \eqref{eq:advection_1d} is simply the simple shift operator $\mathcal{A}u(x) = u(x-a\Delta t)$).
It is followed by the evaluation operator $\mathcal{E}: \mathcal{H} \rightarrow \mathbb{R}^{2m}$ that evaluates values and derivatives at the grid points with the convention that left-sided derivative values are taken at kinks (note that the particular choice of convention does change the method, but it does not compromise its exactness property); and preceded by the interpolation operator $\mathcal{I}: \mathbb{R}^{2m} \rightarrow \mathcal{H}$, defined on each grid cell $[x_j,x_{j+1}]$ by the nonlinear interpolant $P(x)$ introduced above.

At $t=0$, an initial state vector $\boldsymbol{U}^{h,0}$ must be defined that represents an approximation $\mathcal{I}(\boldsymbol{U}^{h,0})$ to the true initial condition $u_0$. Methodologies to do so are discussed in \cref{sec:sampling_IC}. With this setup, conducting many steps of the jet scheme can be interpreted in two equivalent ways:
\begin{equation*}
\text{(a)}\;
S^n = \cdots \; \cdots
\mathcal{E} \mathcal{A} \mathcal{I}
\underbrace{\mathcal{E} \mathcal{A} \mathcal{I}}_{\tilde{\mathcal{N}}}
\underbrace{\mathcal{E} \mathcal{A} \mathcal{I}}_{\tilde{\mathcal{N}}}
(\boldsymbol{U}^{h,0})\;,
\ \text{or}\ \text{(b)}\;
S^n = \cdots \; \cdots
\mathcal{E} \mathcal{A}
\underbrace{\mathcal{I} \mathcal{E} \mathcal{A}}_{\mathcal{N}}
\underbrace{\mathcal{I} \mathcal{E} \mathcal{A}}_{\mathcal{N}}
(\mathcal{I} (\boldsymbol{U}^{h,0}))\;.
\end{equation*}
Here $\mathcal{N} = \mathcal{I} \circ \mathcal{E} \circ \mathcal{A}$ is the induced jet scheme one-step operator acting between functions, i.e., the jet scheme is the true solution operator $\mathcal{A}$ followed by the operator $\mathcal{I}\mathcal{E}$. Note that for the linear jet schemes considered in \cite{SeiboldRosalesNave2012}, the operator $\mathcal{I}\mathcal{E}$ is a projection. This last property is generally not given for the nonlinear jet scheme; however, it is not needed for the characterization of the $t\to\infty$ limits of the jet scheme. To conclude, both the finite-dimensional update rule (a), which represents how the scheme is implemented, as well as the update rule of functions, (b), can be employed for the numerical analysis carried out in \cref{subsec:numerical_analyis_jet_scheme}.

\subsection{Approximation of the initial condition}
\label{sec:sampling_IC}
There are many ways how one can approximate a given initial condition $u_0$ via an initial function that is a jet scheme interpolant. For a given $m=1/h$, denote the set of grid points by $X_m=\{x_1,x_2,\dots,x_m\}$. If $u_0\in C^1$, a natural approach, denoted $(\phi,\psi)$, is the direct evaluation of values and slopes on the grid, and the assignment $\phi_j = u_0(x_j)$ and $\psi_j = u_0'(x_j)$ to define the initial jet scheme approximation. This approach, discussed in \cref{subsec:other_ic_approximation}, turns out to be accurate, but it does not guarantee commuting limits.

For that reason, we now introduce a different approximation strategy of the initial condition, denoted $(\phi^{\delta},\psi^{\delta})$, that only requires $u_0\in C^0$ and that establishes the jet scheme's exactness. First, we evaluate $u_0$ on the grid points to obtain preliminary values $\hat{\phi}_j = u_0(x_j)$ for $j=1\,\dots,m$. Without loss of generality, we can assume that no three consecutive points are collinear (if there is such a triple, then we can modify the position of the middle point with minimal error to make them non-collinear). Based on these data $\{(x_1,\hat{\phi}_1),(x_2,\hat{\phi}_2),\cdots,(x_m,\hat{\phi}_m)\}$, let $I(x)$ be the standard piecewise linear interpolation (that connects the dots, forming kinks at the grid points). We now choose a parameter $\delta$ with $\epsilon\ll\delta \ll h$, where $\epsilon$ is the machine precision, and consider the function $I(x+\delta)$. This is the interpolant $I(x)$ shifted to the left by $\delta$, i.e., its kinks are at $x_j-\delta$ for $j=1\,\dots,m$. Hence, $I(x+\delta)$ has well-defined slopes at the grid points, and we can assign
\begin{equation}
\label{modified_IC}
\phi_j = I(x_j+\delta) = \hat{\phi}_{j} + \frac{\hat{\phi}_{j+1}-\hat{\phi}_{j}}{h} \delta
\quad\text{and}\quad
\psi_j = I'(x_j+\delta) = \frac{\hat{\phi}_{j+1}-\hat{\phi}_{j}}{h}\;.
\end{equation}
Below we show that this particular approximation of the initial condition gives rise to the jet scheme's exactness, thus providing a rigorous way to construct methods that possess the commuting limits property. Afterwards, we discuss the pros and cons of other sampling strategies.

\subsection{Quantifying approximation errors and convergence}
\label{sec:error convergence}
Conceptualizing convergence of numerical methods requires defining a distance between the true solution (a function) and the sequence of numerical approximations (as $h\to 0$). This is commonly done via a function space norm (e.g., $L^p$) as in finite elements, or a \emph{scaled} discrete norm that is induced from a continuum norm as in finite differences \cite{LeVeque2007}. The connection between the $h$-scaled grid norm and the continuum norm is: (a)~the former is a quadrature rule for the latter; and moreover: (b)~the former can be defined by means of the latter via an interpolation.

Because jet schemes are fundamentally based on an interpolation step, this last interpretation comes naturally for them: rather than working with a discrete norm on the data $(\phi_j,\psi_j)$, we directly measure errors in a continuum norm, based on the nonlinear Hermite interpolant defined above. In contrast to the Von Neumann analysis for linear methods in \cref{sec:linear_methods}, we do not employ Fourier modes for the analysis of jet schemes. Hence we use the $L^1$ norm.

An important consequence of measuring errors in a continuum norm is that in general there is a nonzero error between the true initial condition and the initial ($t=0$) interpolant that approximates it. Hence, there are two ways to measure the error of the numerical approximation at some time $t=t_\mathrm{f}$: the \emph{total error}, i.e., the $L^1$-distance of the interpolant at $t=t_\mathrm{f}$ to the true solution; and the \emph{evolution error}, i.e., the $L^1$-distance of the interpolant at $t=t_\mathrm{f}$ to a proxy of the true solution, namely the interpolant from $t=0$ translated via \eqref{eq:advection_1d} to $t_\mathrm{f}$. We have already encountered this important distinction in \cref{sec:linear_methods}, and it is equally important for jet schemes.

The right panel of \cref{fig:error_convergence} presents a convergence study of the jet scheme error $\varepsilon^{h,t_\mathrm{f}}$ as function of $h$ and $t_\mathrm{f}$. The light blue/red curves show the total error (at different final times), while the dark blue/red curves display the pure evolution errors (blue vs.\ red are the two different ways to approximate the initial conditions, see \cref{sec:sampling_IC}). The results show that the total error is generally dominated by the approximation error of the initial condition (the $(\phi,\psi)$ strategy yields a factor 2 smaller errors, but the $(\phi^{\delta},\psi^{\delta})$ strategy does not require $u_0$ to be differentiable), while the scheme's evolution error is much smaller.

One can show that the best one can do to approximate a general smooth initial condition by a piecewise linear function is second order (in the number of kinks), so the jet scheme cannot be better than second order if the total error is considered. However, as we are interested in the (loss of) accuracy as $t \to \infty$ of the scheme itself, we focus on the evolution error.

Because the evolution error is the $L^1$-distance between two jet scheme interpolants that are piecewise linear, this error can be calculated up to machine precision by (i)~forming sub-intervals bounded by the kinks (of both interpolants) and the intersection points between the interpolants, and (ii)~integrating the difference on each sub-interval exactly. In contrast, the total error cannot be calculated exactly (for general $u_0$), but it can be approximated very accurately by forming similarly splitting grid cells into sub-intervals and then using Gaussian quadrature on each sub-interval (we find that a 3-point Gauss rule suffices here).

Considering the evolution error in the right panel of \cref{fig:error_convergence} one can observe that the special $(\phi^{\delta},\psi^{\delta})$ approximation strategy \eqref{modified_IC} turns out to yield an exact method by means of creating fixed point solutions (see \cref{subsec:numerical_analyis_jet_scheme}) right from the get go (the error (dark red) is essentially machine precision times the number of time steps taken). In contrast, the $(\phi,\psi)$ approximation strategy yields smaller total errors, but the evolution errors (dark blue) are larger than with $(\phi^{\delta},\psi^{\delta})$. However, they do decay rapidly as $h\to 0$. Moreover, the scheme appears to be exact nevertheless (the errors do not (up to machine precision) grow with $t_\mathrm{f}$), albeit not via fixed point solutions right away. Comparing those results to the Lax-Wendroff method (left panel of \cref{fig:error_convergence}), we can see that the exactness property of the jet scheme results in substantially smaller errors, particularly for large final times (the black reference line is the same line in both panels).

\subsection{Numerical analysis for nonlinear jet scheme}
\label{subsec:numerical_analyis_jet_scheme}
Thus far, the exactness of the jet scheme has been based on numerical evidence. Now we establish the exactness of the methods with the strategy \eqref{modified_IC} via rigorous arguments. Moreover, we provide a pathway towards a novel numerical analysis for the exactness of numerical methods (in one space dimension). Here we assume that the CFL number $\mu \in \mathbb{Q}$, and let $\mu=\frac{p}{q}$ with $p$ and $q$ co-prime. It is numerically observed that the restriction $\mu \in \mathbb{Q}$ can be relaxed, but for the analysis of the jet schemes done here, we stick to this mild assumption.
We exploit the connection between the two interpretations of jet schemes, established in \cref{subsec:two_jet_interpretations}, and employ the function space formulation of the jet scheme to prove its exactness. 
On the periodic domain $[0,1]$ with CFL number $0<\mu<1$ and $h=1/m$, we have $\Delta t = \frac{\mu h}{a}$. Hence, after $N = \frac{1}{\mu h}$ time steps ($\text{as} \ \mu \in \mathbb{Q}$, we can choose $m$ such that $N \in \mathbb{N}$) the true solution has traveled a distance $1$ and thus has returned to its initial configuration, i.e., $u(x,t+\frac{1}{a}) = u(x,t)$ considering periodicity. We consider the $N$-step mapping from the function space $\mathcal{H}$ defined in \cref{subsec:two_jet_interpretations} into itself
\begin{equation*}
{\Psi}^{h} = \mathcal{N}^N : \mathcal{H}\rightarrow \mathcal{H}
\end{equation*}
that is an approximation to the identity map. Because we are now concerned with the $h\to 0$ convergence of the family of methods, we now use the superscript $h$ to denote the maps' and state vectors' dependence on the resolution. As motivated above, our goal is to characterize the functions that are \emph{fixed points} of ${\Psi}^h$, because any such function $\mathcal{I}(\boldsymbol{U}^{h})$, where $\boldsymbol{U}^{h} \in \mathbb{R}^{2m}$, that satisfies ${\Psi}^h \mathcal{I}(\boldsymbol{U}^{h}) = \mathcal{I}(\boldsymbol{U}^{h})$ constitutes an \emph{exact solution}. Moreover, given the family of maps $\Psi^h$, parameterized by $h$, we have an \emph{exact method} if for a given true solution $u_0(x)$, a sequence $\{\mathcal{I}(\boldsymbol{U}^{h})\}_{h\to 0}$ of fixed points (i.e., $\Psi^h \mathcal{I}(\boldsymbol{U}^{h}) = \mathcal{I}(\boldsymbol{U}^{h})$) exists that converges (in $L^1$) to $u_0$ as $h\to 0$.

Now we provide a condition for a function $\mathcal{I}(\boldsymbol{U}^{h})$ to be a fixed point (of the map ${\Psi}^h$) in terms of its kinks' positions. The function $\mathcal{I}(\boldsymbol{U}^{h})$ can have at most one kink in the interior of each cell, and there could be kinks on the boundary of the cell. We divide each of the $m$ cells into $q$ sub-cells of size $h/q$, and we consider these resulting $mq$ sub-cells from left to right. We can now determine the positions of all the kinks of the function $\mathcal{I}(\boldsymbol{U}^{h})$. For kinks that fall on the boundary between adjacent sub-cells, we use the consistent convention to associate them to the left of those two sub-cells.

As established in \cref{subsec:two_jet_interpretations}, the one-step map $\mathcal{N}$ amounts to first shifting the function $\mathcal{I}(\boldsymbol{U}^{h})$ to the right by $\mu h$ via the true solution operator $\mathcal{A}$, followed by the operator $\mathcal{I}\mathcal{E}$ that replaces the shifted function by the jet scheme interpolant based on its values and slopes evaluated at the grid points. Due to the construction of the jet scheme interpolant, the application of $\mathcal{I}\mathcal{E}$ leaves the shifted interpolant $\mathcal{I}(\boldsymbol{U}^{h})(x-\mu h)$ from the prior step unchanged, except on those grid cells that contain two kinks (it is impossible for the shifted interpolant to have more than two kinks per cell). On those cells, the application of $\mathcal{I}\mathcal{E}$ replaces the shifted interpolant by a new genuine piecewise linear function with a single kink in the cell, or by the fallback function.

A key mechanism to finding fixed points of the $N$-step mapping ${\Psi}^{h} = \mathcal{N}^N$ is to characterize functions for which such discrete changes in the approximating function do not occur. We define:
\begin{definition}
A piecewise linear jet scheme interpolant $\mathcal{I}(\boldsymbol{U}^{h})$, defined by a state vector $\boldsymbol{U}^{h}$, is called \textbf{\emph{defective}} if there exists at least one adjacent pair of kinks that have less than $q-1$ sub-cells in between them. Otherwise, it is called \textbf{\emph{non-defective}}.
\end{definition}
\begin{lemma}
If $\mathcal{I}(\boldsymbol{U}^{h})$ is non-defective then it is a fixed point of the mapping ${\Psi}^{h}$.
\end{lemma}
\begin{proof}
As $\mathcal{I}(\boldsymbol{U}^{h})$ is non-defective, it has at least $q-1$ sub-cells in between any pair of adjacent kinks. Hence, it is impossible to ever generate a situation where two kinks fall within the same grid cell, if the function $\mathcal{I}(\boldsymbol{U}^{h})$ is only shifted by integer multiples of the sub-cell width $h/q$. And in fact, that is the case, because by construction each update step $\mathcal{N}$ shifts $\mathcal{I}(\boldsymbol{U}^{h})$ by $\mu h = p \frac{h}{q}$, i.e., by exactly $p$ sub-cells.
\end{proof}

Based on this intuitive geometric characterization of an important class of fixed points of the jet scheme, we can easily analyze the two initial condition approximation strategies described in \cref{sec:sampling_IC}. While the $(\phi,\psi)$ approach generally generates defective interpolants (e.g., two adjacent grid cells may have kinks very close to their shared grid point), the $(\phi^{\delta},\psi^{\delta})$ method, defined in \eqref{modified_IC}, always generates non-defective interpolants, because adjacent kinks are precisely a distance $h$ apart from each other.

\begin{figure}[htb]
\centering
\begin{tikzpicture}[scale=0.75]
\draw [thick](-6,0) -- (6,0);
\filldraw [black] (-6,0) circle (2pt); 
\draw [thick,dashed](-6,-0.2) -- (-6,0.5);
\filldraw [black] (-2,0) circle (2pt); 
\draw [thick,dashed](-2,-0.2) -- (-2,0.5);
\filldraw [black] (2,0) circle (2pt); 
\draw [thick,dashed](2,-0.2) -- (2,0.5);
\filldraw [black] (6,0) circle (2pt); 
\draw [thick,dashed](6,-0.2) -- (6,0.5);
\node[below =2pt of {(-6,0)}] {$x_0=0$};
\node[below =2pt of {(-2,0)}] {$x_1$};
\node[below =2pt of {(2,0)}] {$x_2$};
\node[below =2pt of {(6,0)}] {$x_3=1$};
\node[above =2pt of {(-5,0)}] {$x_\mathrm{ft}$};
\filldraw [blue] (-5,0) circle (2pt);
\filldraw [red] (-4,0) circle (2pt);
\filldraw [red] (-3,0) circle (2pt);
\filldraw [blue] (-1,0) circle (2pt);
\filldraw [red] (0,0) circle (2pt);
\filldraw [red] (1,0) circle (2pt);
\filldraw [blue] (3,0) circle (2pt);
\filldraw [red] (4,0) circle (2pt);
\filldraw [red] (5,0) circle (2pt);
\filldraw [black!60!green] (-2.4,0) circle (2pt);
\node[above =2pt of {(-2.4,0)}] {$x_\text{k}$};
\filldraw [black!60!green] (1.6,0) circle (2pt);
\draw [thick,dashed](1.6,-0.2) -- (1.6,0.5);
\draw [thick](1.6,0.5) -- (2,0.5);
\filldraw [black!60!green] (5.6,0) circle (2pt);
\draw[thick, ->] (-2.55,0.6) arc (0:180:0.4);
\node[above =2pt of {(1.8,0.4)}] {$\delta$};
\end{tikzpicture}
\vspace{-.5em}
\caption{Domain with three cells. Black dots denote the grid points; blue dots denote the foot points; red dots are the sub-cell boundaries; green denotes the kink positions that are $\delta \ll h$ distance away from the grid points. The black arrow shows the next relative position of kink.}
\label{fig:domain_sub-cells}
\end{figure}
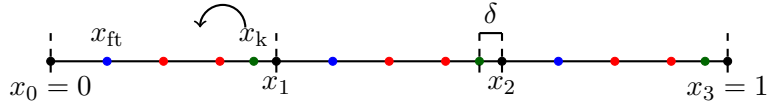

As an illustration, consider $m = 3$ grid cells with CFL number $\mu = \frac{3}{4}$, i.e., traveling the distance $1$ requires $N = 4$ steps. The $(\phi^{\delta},\psi^{\delta})$ strategy produces a piecewise linear function whose kinks are all a distance $\delta$ left of a grid point; hence they all fall into the rightmost sub-cell of each grid cell. In turn, the foot points fall on the boundary between the first and second sub-cell, see \cref{fig:domain_sub-cells}. Hence, $N = 4$ steps of the jet scheme amount to the sequence $[\text{L} \ \text{L} \ \text{L} \ \text{R}]$, where (L)/(R) means that in each cell, the left/right-sided line of the jet scheme interpolant is used.

We are now in a position to prove the main theorem, establishing the exactness of the jet scheme with the specific approximation of the initial condition \cref{modified_IC}.
\begin{theorem}[Exactness]
\label{exactness_thm}
\Cref{numerical_method} with the strategy \cref{modified_IC} of approximating the periodic initial function $u_{0}(x)\in C^{0}([0,1])$, applied to problem \cref{eq:advection_1d}, has the commuting limits property \cref{commuting_limits}.
\end{theorem}
\begin{proof}
The initial vector $\boldsymbol{U}^{h,0} = [\phi_1^{0},\psi_1^{0},\phi_2^{0},\psi_2^{0}, \dots ,\phi_m^{0},\psi_m^{0}]^{T}$ in \cref{modified_IC} induces a piecewise linear interpolant $\mathcal{I}(\boldsymbol{U}^{h,0})$ which is exactly equal to $I(x+\delta)$ in \cref{modified_IC}, where $I(x)$ is the standard piecewise linear function that connects the grid point values. We have already shown that $\mathcal{I}(\boldsymbol{U}^{h,0})$ is a fixed point of the jet scheme. So for a given $h$, the numerical solution at any time step $t_n$ satisfies $\mathcal{I}(\boldsymbol{U}^{h,t_n})(x) = \mathcal{I}(\boldsymbol{U}^{h,0})(\textrm{mod}(x-n \mu h,1))$. As established in \cref{sec:error convergence}, we consider the jet scheme's total error as the $L^1$-distance of the approximate function from the true solution:
\begin{align*}
\boldsymbol{\varepsilon}^{h,t_n}
&= \|\mathcal{I}(\boldsymbol{U}^{h,t_n})-\boldsymbol{u}^{t_n}\|_{{L}^1}\\
&= \|\mathcal{I}(\boldsymbol{U}^{h,0})(\textrm{mod}(x-n \mu h,1))-u_0(\textrm{mod}(x-n \mu h,1))\|_{{L}^1}\\
&= \|\mathcal{I}(\boldsymbol{U}^{h,0})(x)-u_0(x)\|_{{L}^1} = \|I(x+\delta)-u_0(x)\|_{{L}^1}\\
& \leq \|I(x+\delta)-u_0(x+\delta)\|_{{L}^1} +
\|u_0(x+\delta)-u_0(x)\|_{{L}^1} \\
& \leq \boldsymbol{\omega}(u_0,h)+\boldsymbol{\omega}(u_0,\delta)\;,
\end{align*}
where $\boldsymbol{\omega}(u_0,r) = \text{sup}\{|u_0(x)-u_0(y)|:x,y\in[0,1], \ |x-y|<r\}$ is the modulus of continuity of $u_0 \in C^{0}([0,1])$. The last inequality arises as follows. On each sub-interval $[x_j,x_{j+1}]$, one has $\int_{x_j}^{x_{j+1}}|I(x)-u_0(x)|\,\textrm{d} x \leq h \boldsymbol{\omega}(u_0,h)$, and thus $\|I(x+\delta)-u_0(x+\delta)\|_{{L}^1} \leq \boldsymbol{\omega}(u_0,h)$; and by similar arguments $\|u_0(x+\delta)-u_0(x)\|_{{L}^1} \leq \boldsymbol{\omega}(u_0,\delta) $. Consequently, with $\delta \ll h$ (assuming exact arithmetic here), we obtain that
\begin{equation*}
\lim_{h\to 0}\limsup_{n\to\infty}\boldsymbol{\varepsilon}^{h,t_n}
\leq \lim_{h\to 0} [\boldsymbol{\omega}(u_0,h)+\boldsymbol{\omega}(u_0,\delta)] \leq \lim_{h\to 0} [\boldsymbol{\omega}(u_0,h)+\boldsymbol{\omega}(u_0,h)] = 0\;,
\end{equation*}
where $\boldsymbol{\omega}(u_0,\delta) \leq \boldsymbol{\omega}(u_0,h)$ follows from the $r$-monotonicity of $\boldsymbol{\omega}(u_0,r)$, and $u_0$ uniformly continuous on $[0,1]$ implies $\lim_{h \to 0}\boldsymbol{\omega}(u_0,h) = 0$.
Moreover, because the scheme is convergent (in the traditional sense), we also have $\limsup_{n\to\infty} \lim_{h\to 0} \boldsymbol{\varepsilon}^{h,t_n} = 0$. Hence, we have established the commuting limits property
\begin{align*}
\lim_{h\to 0}\limsup_{n\to\infty} \boldsymbol{\varepsilon}^{h,t_n}
= 0 =\limsup_{n\to\infty}\lim_{h\to 0} \boldsymbol{\varepsilon}^{h,t_n}\;.
\end{align*}
\end{proof}

\begin{remark}
Obviously, in floating point arithmetic, there are practical limitations to the just established convergence notions. As in \cref{sec:sampling_IC}, one should choose $\epsilon\ll\delta \ll h$ (where $\epsilon$ is the machine precision); and the results in \cref{fig:error_convergence} indicate that a linear accumulation of round-off errors could occur. Both aspects imply technical limitations on the $h\to 0$ as well as $t_\mathrm{f}\to\infty$ limits; however, those are not much different from what is incurred in other numerical methods. Moreover, the results in \cref{fig:error_convergence} indicate that for practically relevant choices of $h$ and $t_\mathrm{f}$, there are no detriments to the exactness of the jet schemes.
\end{remark}

\subsection{Other initial condition approximation: good news and some caveats}
\label{subsec:other_ic_approximation}
With \cref{exactness_thm} we have established the exactness of jet schemes, if the $(\phi^{\delta},\psi^{\delta})$ initial function approximation \eqref{modified_IC} is employed. The underlying reason is that the initial function $\mathcal{I}(\boldsymbol{U}^{h,0})$ is a fixed point of the jet scheme update step. However, other approximations of the initial conditions may yield smaller error constants. For instance, \cref{fig:error_convergence} shows that the $(\phi,\psi)$ strategy, which is natural if $u_0\in C^1$ as it is based on direct evaluations of the initial condition and its derivatives, turns out to generate second order convergent approximations with smaller errors than the $(\phi^{\delta},\psi^{\delta})$ strategy. A key drawback is that $(\phi,\psi)$ strategy generally yields an initial vector $\boldsymbol{U}^{h,0} \in \mathbb{R}^{2m}$, or function $\mathcal{I}(\boldsymbol{U}^{h,0})$, that is \emph{not} a fixed point of the jet scheme. However, numerical evidence (\cref{fig:error_convergence} and below) suggests that, for a wide range of sufficiently smooth initial functions, the jet scheme does always assume a fixed point after a finite number of update steps, and thus nevertheless gives rise to the commuting limits property \cref{commuting_limits}.

To demonstrate this property, and its consequences, more distinctly, we present a numerical test of the jet scheme with the $(\phi,\psi)$ strategy up to a really large final time $t_\textrm{f} = 100000$. \cref{fig:jet_scheme_constant_coefficients} shows the numerical results on $m = 100$ grid points with CFL number $\mu = 0.9$ at $t=100$, $t=1000$, and the final time. One can see that the jet scheme solution persists up to infinite time. Moreover, we observe (not shown here) that if the computational grid is refined, one always obtains exact solutions after finitely many steps, and they converge (as $h\to 0$) to the true solution. In other words, this scheme is exact for this problem.

Moreover, we compare the long-time evolution of jet schemes with other nonlinear (we already know from \cref{sec:linear_methods} that linear methods have no chance of converging at $t\to\infty$) classical high resolution methods: fifth-order WENO \cite{jiang1996efficient}, Lax-Wendroff with van Leer limiter \cite{LeVeque2002}, and Lax Wendroff with Superbee limiter \cite{LeVeque2002}. In \cref{fig:jet_scheme_constant_coefficients}, the three snapshots in time are complemented (bottom right panel) by the temporal evolution of the maximum of the numerical solutions. Clearly, all classical schemes keep on deforming in time, with the WENO and van Leer solutions being essentially constant at $t_\textrm{f}$, and the Superbee solution flattening on a very slow time scale. These results demonstrate the jet scheme's superiority over the classical methods in terms of its long-time numerical behavior for constant-coefficient advection problems.

\begin{figure}[ht]
\begin{minipage}[b]{.49\textwidth}
\includegraphics[width=\textwidth]{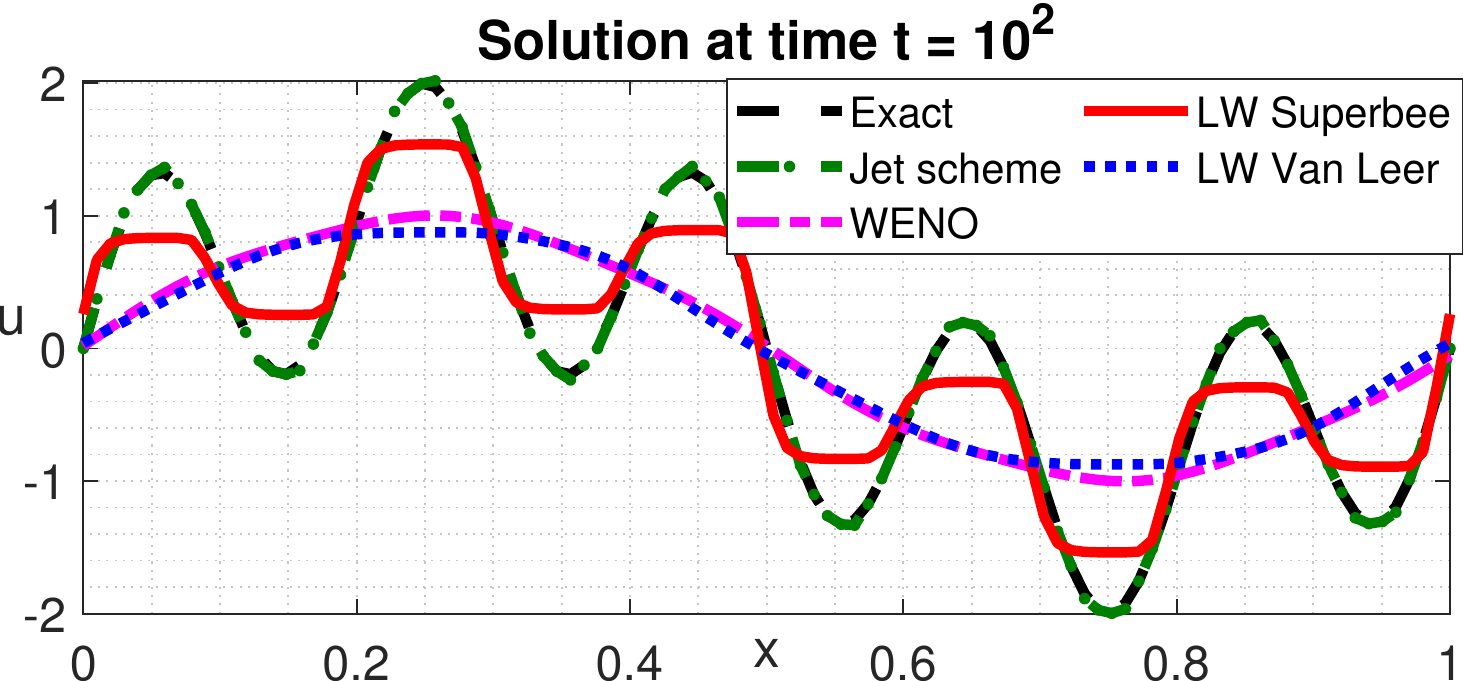}
\end{minipage}
\hfill
\begin{minipage}[b]{.49\textwidth}
\includegraphics[width=\textwidth]{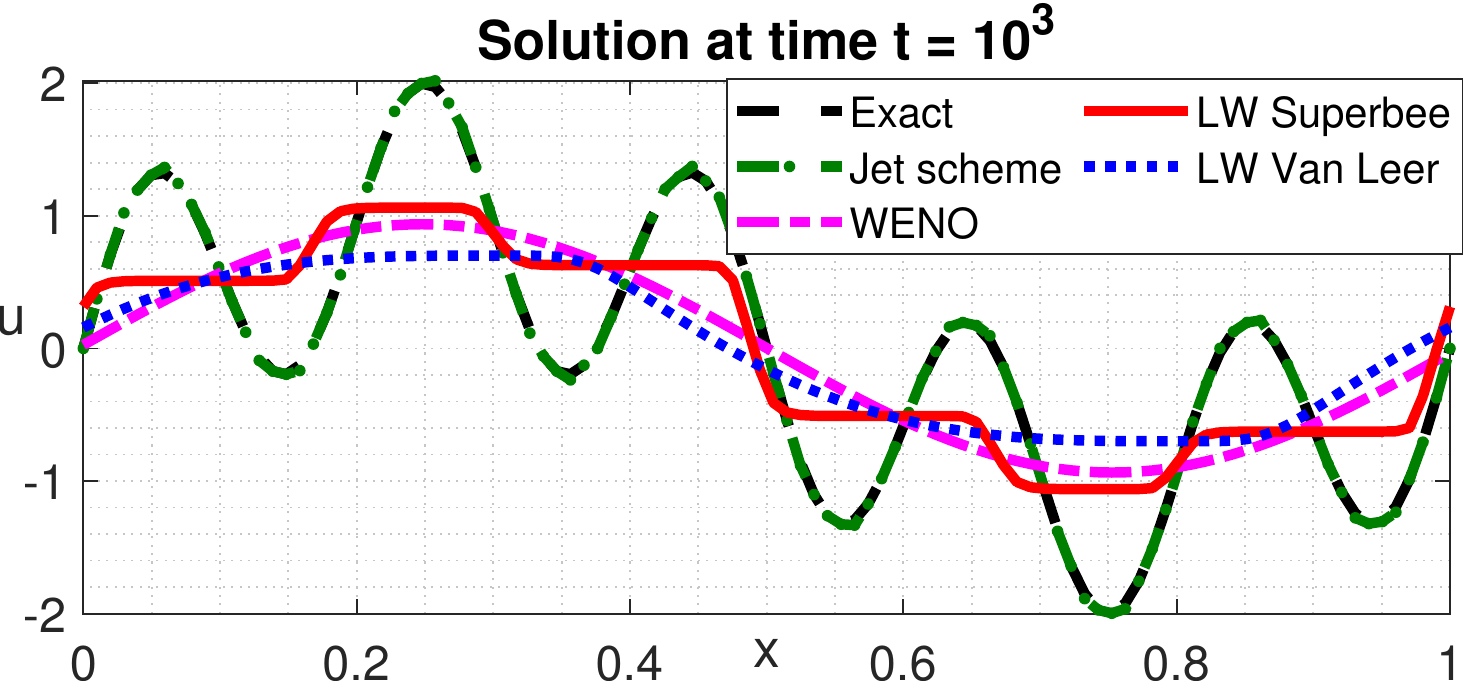}
\end{minipage}

\begin{minipage}[b]{.49\textwidth}
\includegraphics[width=\textwidth]{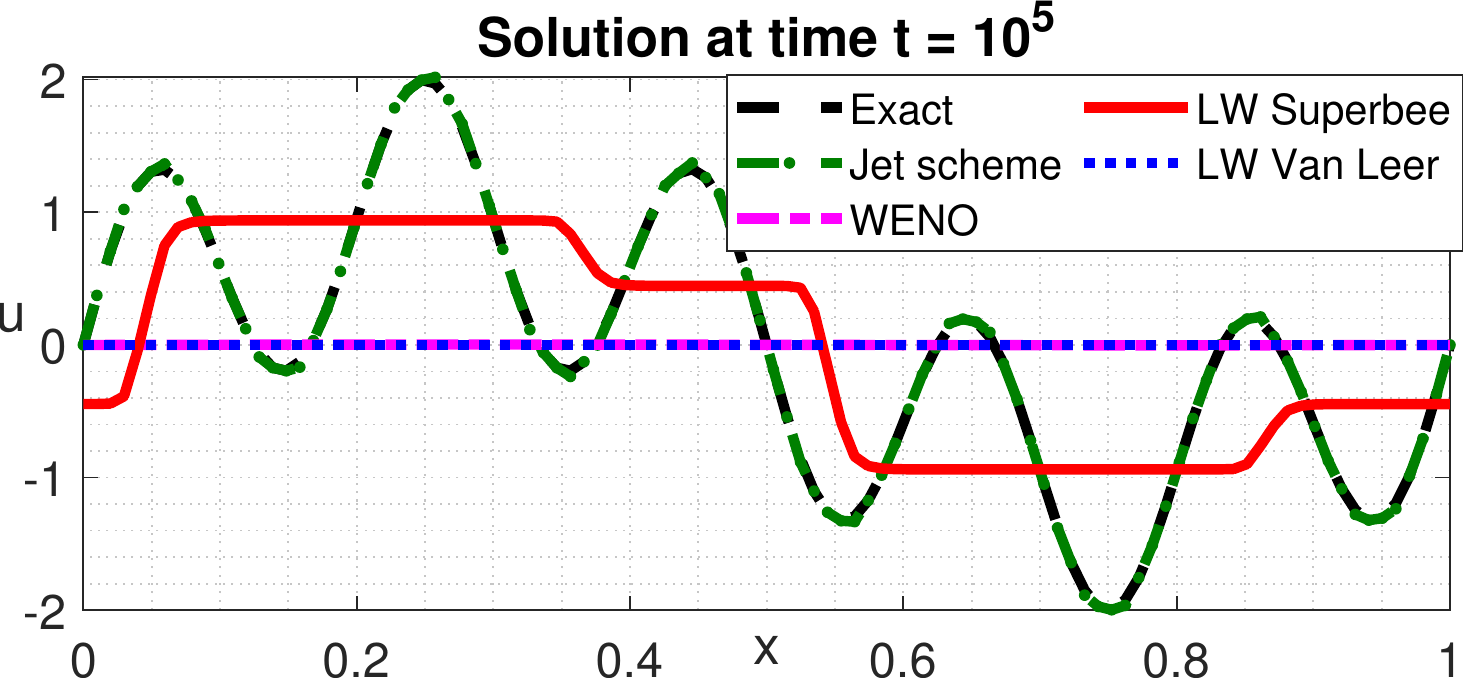}
\end{minipage}
\hfill
\begin{minipage}[b]{.49\textwidth}
\includegraphics[width=\textwidth]{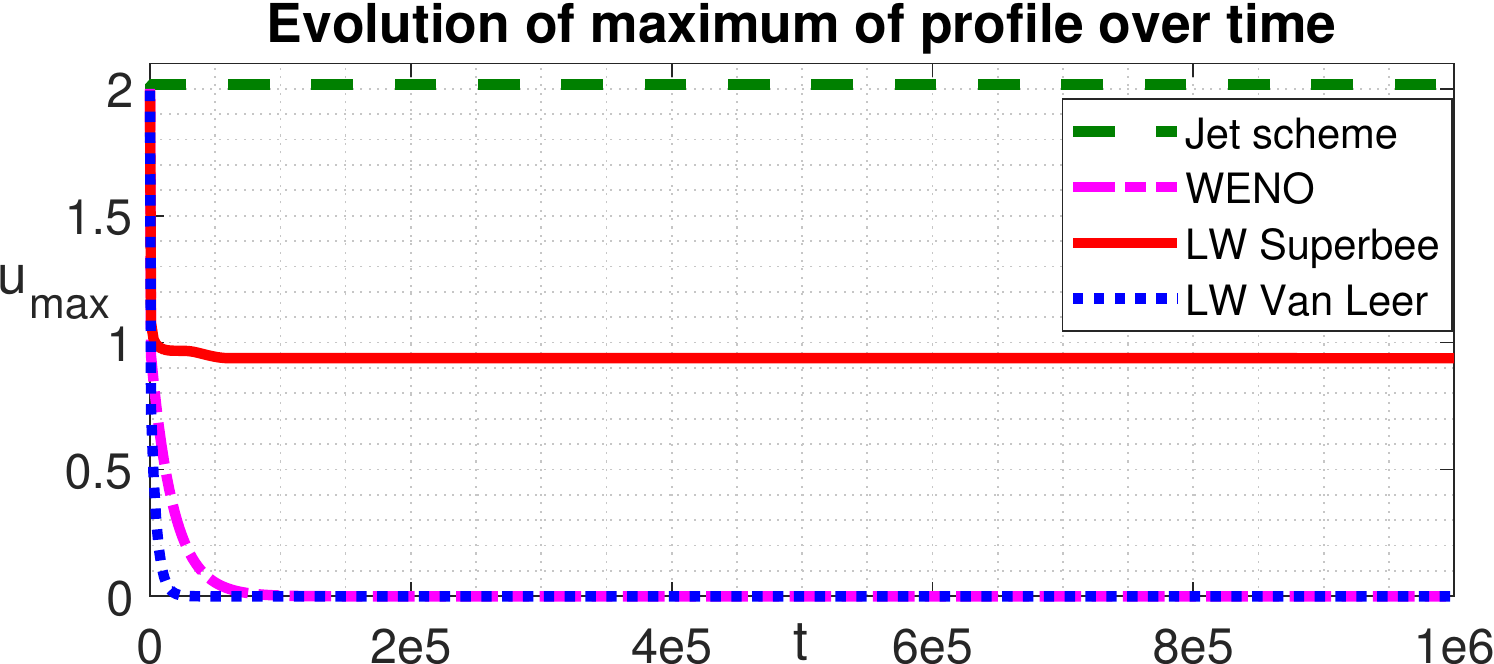}
\end{minipage}
\caption{Example of long-time evolution of nonlinear jet scheme with $(\phi,\psi)$ initial condition approximation, for $u_t+u_x=0$ with $\mu=0.9$ on $m=100$ grid points, in comparison with three classical nonlinear methods ('LW' stands for \textrm{Lax-Wendroff}). While the classical methods yield deformed long-time solutions, the jet scheme assumes (after finitely many steps) a fix point (i.e., an exact solution) that is a good approximation of the true solution. The bottom right panel shows the temporal evolution of the maximum, and it confirms the jet scheme's superiority over the existing methods in capturing long-time evolution.}
\label{fig:jet_scheme_constant_coefficients}
\end{figure}



While these observations are very promising, it is actually not true that any possible initial vector $\boldsymbol{U}^{h,0}$ will be transformed (by the jet scheme) into a fixed point. Here we establish a procedure to construct examples where the jet scheme does not assume a fixed point after a finite number of time steps. To find those, we seek an eigenvector of the mapping ${\Psi}^h$ with associated eigenvalue of modulus strictly less than $1$. We do so by restricting the map ${\Psi}^h$ to be of the form $M^N$, where $M$ is a specific one-step update rule that is linear. Because in this case the jet scheme applies the same update rule in every time step, it suffices to analyze just one step of the scheme.

Specifically, consider $m = 6$ and $\mu = \frac{3}{4}$, so that after $N = 8$ steps the solution has traveled a distance $1$. Now assume that the updates on the 6 cells are (from left to right, where 'L' denotes the case that the foot point is left to the kink and 'F' denotes the fallback case): [L \ L \ F \ L \ L \ F]. It is straightforward to check that the resulting matrix $M$ has a real eigenvalue with modulus less than $1$. Hence, if we start with the associated eigenvector as initial condition, the numerical solution will exponentially decay to zero, and it will never turn into a fixed point in finite time, as shown in \cref{counter_example}.

\begin{figure}[ht]
\begin{minipage}[b]{.49\textwidth}
\includegraphics[width=\textwidth]{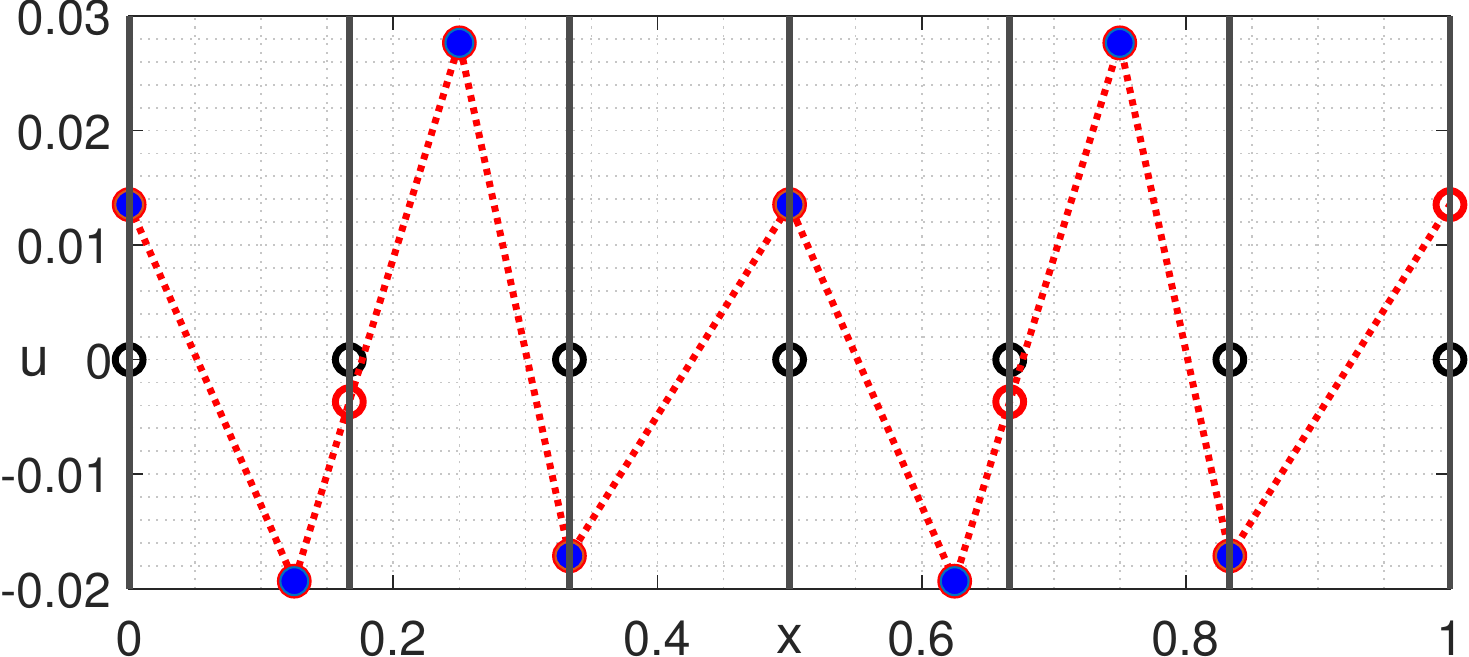}
\end{minipage}
\hfill
\begin{minipage}[b]{.49\textwidth}
\includegraphics[width=\textwidth]{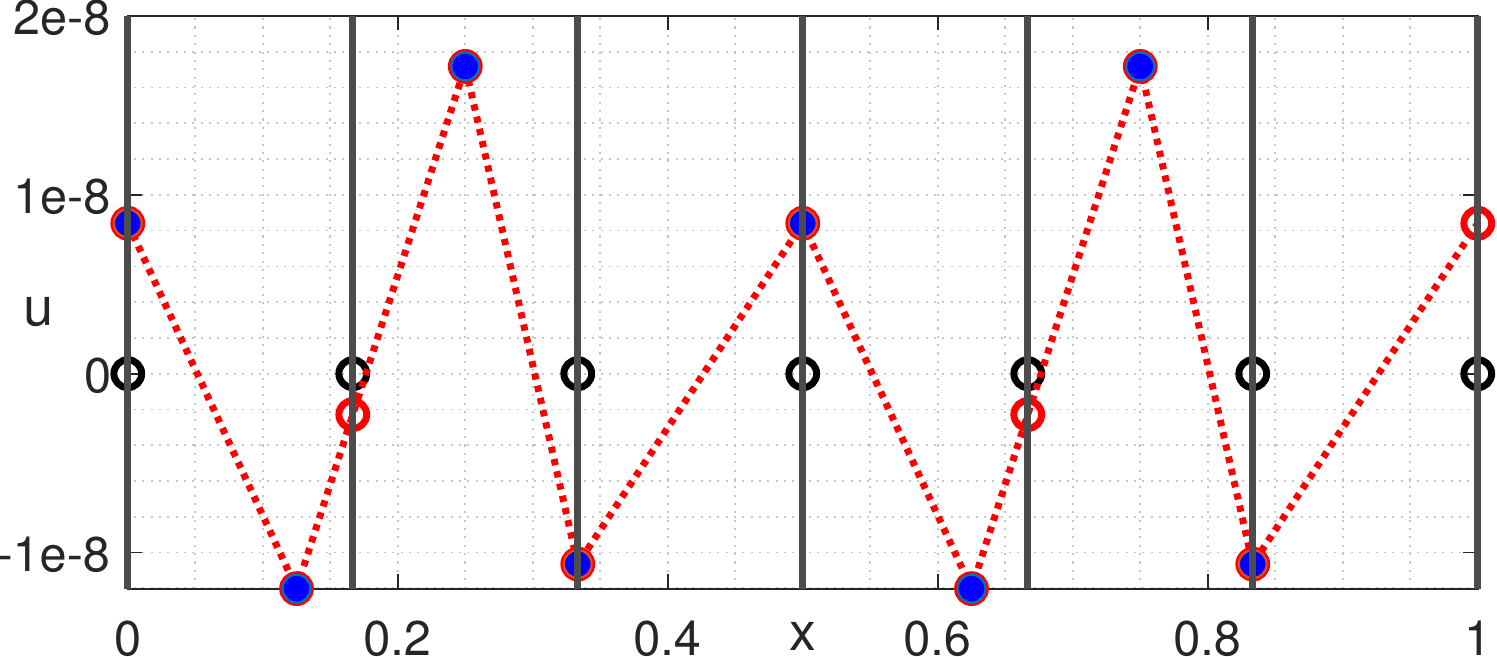}
\end{minipage}
\vspace{-.5em}
\caption{Eigenvector of the one-step matrix $M$ corresponding to an eigenvalue with modulus less than $1$.
Left: initial condition. Right: solution at $t_\mathrm{f}=5$, which is exactly same as the initial profile, but multiplied by $10^{-6}$.}
\label{counter_example}
\end{figure}


Based on this example, we can also construct an infinite family of corresponding examples for larger grid sizes $m$ (multiples of 6), by simply concatenating the sequence [L \ L \ F \ L \ L \ F]. However, such functions are by construction highly oscillatory (on the scale $h$). In contrast, $h\to 0$ sequences of approximations to a smooth initial function are the opposite of such grid-scale oscillatory functions; and in fact, decaying eigenvectors like the ones constructed here never arose (for $h$ sufficiently small) in any test cases we conducted. An intuitive explanation for this fact is as follows. As $h$ becomes small, any smooth function locally is $\mathcal{O}(h^3)$ close to a parabola on a grid cell of size $h$. If one had an exact parabola (with nonzero quadratic part) on that cell, then the associated piecewise linear jet scheme interpolant would have its kink precisely in the center of the interval. Consequently, for the true function, kinks will be $\mathcal{O}(h^3)$ away from the center of each cell, except for those few cells in which the function has inflection points. As a consequence, for $h$ sufficiently small, the approximation of a smooth initial condition will be away from those oscillatory eigenvectors. That being said, even for tiny $h$, one generally cannot expect to start with a fixed point of the jet scheme because near inflection points, fallback cases may arise. However, as discussed above, we do observe that for $h$ sufficiently small, a fixed point is generally assumed after $\mathcal{O}(1)$ steps (independent of $h$).

\vspace{1.5em}
\section{Conclusions and outlook}
\label{sec:conclusions}
The key takeaways from this study are that: (a)~there is merit in conceptualizing infinite time limits of numerical methods for advection problems; and (b)~schemes that possess convergent infinite time limits do exist. Conceptually, the focus on an infinite (or asymptotically large) time horizon is a departure from classical numerical analysis paradigms, and this work establishes new notions of numerical analysis and formalizes the infinite time convergence via the commuting limits property \eqref{commuting_limits} and the concept of exact methods that possess convergent sequences of fixed point solutions. Moreover, in establishing these new concepts, traditional numerical analysis perspectives have been extended, such as considering error convergence in both mesh size $h$ and final time $t_\mathrm{f}$, and the focus on long-time behavior of existing numerical methods.

The specific results and insights found in this work are as follows.
First, it is formally shown that linear methods for advection problems cannot (except for trivial cases) yield exact solutions, and thus they are unable to produce approximations that do not deteriorate in time.
Second, the formal proof is complemented by asymptotic arguments, based on the modified equation, that establish the scaling of the truncation error $\boldsymbol{\varepsilon}^{h,t_\mathrm{f}}$ interpreted as a bivariate function of $h$ and $t_\mathrm{f}$.
Third, the jet scheme methodology with a specific nonlinear interpolant is introduced, including a duality between being an update rule in a finite-dimensional state space vs.\ in a function space.
Fourth, exploiting that duality, we prove that the jet scheme, with a specific initial condition approximation, does indeed satisfy the stringent requirement of producing convergent infinite time solutions.
And finally, the benefit of the jet schemes' infinite time convergence, relative to existing widely used nonlinear methods, is illustrated in a comparative long-term computation example.

There is a broad variety of questions that are triggered by the concepts and results presented here. One important (open) question is whether there are traditional fixed-grid advection schemes that are exact. While we have proven that exactness is impossible to achieve with linear methods, and have demonstrated numerically that various popular nonlinear methods all fail to be exact, this does not mean that there are no such methods. Certainly, the results in \cref{fig:jet_scheme_constant_coefficients} demonstrate that different finite volume limiters may result in vastly different long-time behavior of the numerical solutions. Specifically, while Lax-Wendroff with van Leer limiter clearly produces an essentially constant solution after some large but finite time ($t=10^5$), Lax-Wendroff with Superbee limiter preserves a (deformed but) clearly non-constant profile even after such a long time. This property plausibly has to do with the fact that the Superbee limiter follows the uppermost boundary of the TVD region \cite{LeVeque2002}, and it might hint at the existence of limiters that generate exact methods.

Another important question is to what extent the concepts established herein generalize to more complicated transport problems. Clearly, it is a long route from problem \eqref{eq:advection_1d} to real application problems that may involve (a)~variable coefficients, (b)~higher dimensions, and/or (c)~coupling of transport to other equations. Regarding (a): while the jet scheme methodology can be applied to variable coefficient advection, the exactness properties fail to carry over. At the same time, simple numerical tests indicate that for non-rapidly varying coefficients, it does take a long time for the numerical solution to degrade. Regarding (b): while linear Hermite interpolants generalize naturally to higher space dimensions \cite{SeiboldRosalesNave2012}, there is no such canonical generalization of the piecewise linear interpolant.

Finally, regarding (c), an important question for future research is to what extent exact numerical methods can give rise to more accurate long-time solutions for problems where the accurate resolution of linear advection is a critical sub-problem, such as level set methods in fluid flow simulations or contact discontinuities in compressible gas dynamics. Likewise, one may ask whether some of the presented ideas and concepts can also be extended to other problems with nonlinear traveling waves (solitons, diffusion-reaction, advection-reaction).

\vspace{1.5em}
\section*{Acknowledgments}
This work was supported by the National Science Foundation via grants DMS--1719640 and DMS--2012271.

\vspace{1.5em}
\bibliographystyle{plain}
\bibliography{references}

\begin{thebibliography}{10}

\bibitem{albin2012fourier}
Nathan Albin, Oscar~P Bruno, Theresa~Y Cheung, and Robin~O Cleveland.
\newblock Fourier continuation methods for high-fidelity simulation of
  nonlinear acoustic beams.
\newblock {\em The Journal of the Acoustical Society of America},
  132(4):2371--2387, 2012.

\bibitem{BergerOliger1984}
M.~J. Berger and J.~Oliger.
\newblock Adaptive mesh refinement for hyperbolic partial differential
  equations.
\newblock {\em J. Comput. Phys.}, 53(3):484--512, 1984.

\bibitem{Brunner2002}
T.~A. Brunner.
\newblock Forms of approximate radiation transport.
\newblock Technical Report SAND2002--1778, Sandia National Laboratories, July
  2002.

\bibitem{bruno2010high}
Oscar~P Bruno and Mark Lyon.
\newblock High-order unconditionally stable fc-ad solvers for general smooth
  domains i. basic elements.
\newblock {\em Journal of Computational Physics}, 229(6):2009--2033, 2010.

\bibitem{Chandrasekhar1960}
S.~Chandrasekhar.
\newblock {\em Radiative transfer}.
\newblock Dover, 1960.

\bibitem{ChidyagwaiNaveRosalesSeibold2012}
P.~Chidyagwai, J.-C. Nave, R.~R. Rosales, and B.~Seibold.
\newblock A comparative study of the efficiency of jet schemes.
\newblock {\em Int. J. Numer. Anal. Model.-B}, 3(3):297--306, 2012.

\bibitem{CockburnKarniadakisShu2000}
B.~Cockburn, G.~E. Karniadakis, and C.-W. Shu.
\newblock {\em The development of discontinuous Galerkin methods}.
\newblock Springer, New York, 2000.

\bibitem{CockburnShu2001}
B.~Cockburn and C.-W. Shu.
\newblock {R}unge-{K}utta discontinuous {G}alerkin methods for
  convection-dominated problems.
\newblock {\em J. Sci. Comput.}, 16(3):173--261, 2001.

\bibitem{CourantFriedrichsLewy1928}
R.~Courant, K.~Friedrichs, and H.~Lewy.
\newblock {\"U}ber die partiellen {D}ifferenzengleichungen der mathematischen
  {P}hysik.
\newblock {\em Mathematische Annalen}, 100(1):32--74, 1928.

\bibitem{FarjounSeibold2009}
Y.~Farjoun and B.~Seibold.
\newblock An exactly conservative particle method for one dimensional scalar
  conservation laws.
\newblock {\em J. Comput. Phys.}, 228(14):5298--5315, 2009.

\bibitem{Godunov1959}
S.~K. Godunov.
\newblock A difference scheme for the numerical computation of a discontinuous
  solution of the hydrodynamic equations.
\newblock {\em Math. Sbornik}, 47:271--306, 1959.

\bibitem{GottliebShuTadmor2001}
S.~Gottlieb, C.-W. Shu, and E.~Tadmor.
\newblock Strong stability preserving high order time discretization methods.
\newblock {\em SIAM Rev.}, 43(1):89--112, 2001.

\bibitem{Henshaw1994}
W.~D. Henshaw.
\newblock A fourth-order accurate method for the incompressible
  {N}avier-{S}tokes equations on overlapping grids.
\newblock {\em J. Comput. Phys.}, 113(6):13--25, 1994.

\bibitem{jiang1996efficient}
Guang-Shan Jiang and Chi-Wang Shu.
\newblock Efficient implementation of weighted eno schemes.
\newblock {\em Journal of computational physics}, 126(1):202--228, 1996.

\bibitem{LeVeque2002}
R.~J. LeVeque.
\newblock {\em Finite volume methods for hyperbolic problems}.
\newblock Cambridge University Press, first edition, 2002.

\bibitem{LeVeque2007}
R.~J. LeVeque.
\newblock {\em Finite difference methods for ordinary and partial differential
  equations: {S}teady-state and time-dependent problems}.
\newblock Society for Industrial and Applied Mathematics, first edition, 2007.

\bibitem{Monaghan2005}
J.~J. Monaghan.
\newblock Smoothed particle hydrodynamics.
\newblock {\em Rep. Prog. Phys.}, 68(8):1703--1759, 2005.

\bibitem{NaveRosalesSeibold2010}
J.-C. Nave, R.~R. Rosales, and B.~Seibold.
\newblock A gradient-augmented level set method with an optimally local,
  coherent advection scheme.
\newblock {\em J. Comput. Phys.}, 229(10):3802--3827, 2010.

\bibitem{OsherFedkiw2002}
S.~Osher and R.~P. Fedkiw.
\newblock {\em Level set methods and dynamic implicit surfaces}.
\newblock Springer, New York, 2002.

\bibitem{Salac2011}
D.~Salac.
\newblock The augmented fast marching method for level set reinitialization.
\newblock {\em preprint}, 2011.

\bibitem{SeiboldRosalesNave2012}
B.~Seibold, R.~R. Rosales, and J.-C. Nave.
\newblock Jet schemes for advection problems.
\newblock {\em Discrete Contin. Dyn. Syst. Ser. B}, 17(4):1229--1259, 2012.

\bibitem{ShuOsher1988}
C.-W. Shu and S.~Osher.
\newblock Efficient implementation of essentially non-oscillatory
  shock-capturing schemes.
\newblock {\em J. Comput. Phys.}, 77:439--471, 1988.

\bibitem{trefethen2000spectral}
Lloyd~N Trefethen.
\newblock {\em Spectral methods in MATLAB}, volume~10.
\newblock Siam, 2000.

\bibitem{VanLeer1973}
B.~van Leer.
\newblock Towards the ultimate conservative difference scheme {I}. {T}he quest
  of monotonicity.
\newblock {\em Springer Lecture Notes in Physics}, 18:163--168, 1973.

\bibitem{VanLeer1977_2}
B.~van Leer.
\newblock Towards the ultimate conservative difference scheme {IV}. {A} new
  approach to numerical convection.
\newblock {\em J. Comput. Phys.}, 23:276--299, 1977.

\end{thebibliography}

\vspace{2.5em}
\end{document}